\documentclass[11pt]{amsart}
\usepackage[T1]{fontenc}
\usepackage[utf8]{inputenc}
\usepackage{lmodern}
\usepackage[a4paper,margin=1in]{geometry}
\usepackage{amsmath,amssymb,mathtools,mathrsfs}
\usepackage{array,booktabs,longtable,tabularx}
\usepackage{microtype}
\usepackage{xcolor}
\usepackage{listings}
\usepackage[colorlinks=true,linkcolor=blue!45!black,citecolor=blue!45!black,urlcolor=blue!45!black]{hyperref}
\usepackage{xurl}
\newcommand{\C}{\mathbf C}
\newcommand{\Q}{\mathbf Q}
\newcommand{\Z}{\mathbf Z}
\newcommand{\F}{\mathbf F}
\newcommand{\OO}{\mathcal O}
\newcommand{\pp}{\mathfrak p}
\newcommand{\Tr}{\operatorname{Tr}}
\newcommand{\Nm}{\operatorname{N}}
\newcommand{\Gal}{\operatorname{Gal}}
\newcommand{\Res}{\operatorname{Res}}
\newcommand{\floor}[1]{\left\lfloor#1\right\rfloor}
\newcommand{\ceil}[1]{\left\lceil#1\right\rceil}
\DeclareRobustCommand{\code}[1]{{\small\nolinkurl{#1}}}
\newcommand{\codeexpr}[1]{{\small\texttt{#1}}}
\newcommand{\SMLbase}{https://github.com/fukubillueda-web/epsilon_SML/blob/5b55d3f6e7c053f733afb1be523e532d212d94d0}
\newcommand{\FMLbase}{https://github.com/fukubillueda-web/epsilon_FML/blob/b91ad93900a5e6d9207a4e6166ac94bb3bbcf3ed}
\ExplSyntaxOn
\tl_new:N \l_sml_decl_text_tl
\tl_new:N \l_sml_decl_identifier_tl
\cs_new_protected:Npn \sml_decl_identifier:n #1
  {
    \group_begin:
    \int_compare:nNnTF { \tl_count:n {#1} } > { 16 }
      {
        \tl_set:Nn \l_sml_decl_identifier_tl {#1}
        \regex_replace_all:nnN { ([a-z0-9])([A-Z]) }
          { \1\c{allowbreak}\2 } \l_sml_decl_identifier_tl
        \regex_replace_all:nnN { ([A-Z])([A-Z][a-z]) }
          { \1\c{allowbreak}\2 } \l_sml_decl_identifier_tl
        \regex_replace_all:nnN { (_) }
          { \1\c{allowbreak} } \l_sml_decl_identifier_tl
        \int_compare:nNnT { \tl_count:n {#1} } > { 24 }
          {
            \regex_replace_all:nnN { ([A-Za-z0-9]) }
              { \1\c{penalty}200\c{relax} } \l_sml_decl_identifier_tl
          }
        \tl_use:N \l_sml_decl_identifier_tl
      }
      { \mbox{#1} }
    \group_end:
  }
\cs_new_protected:Npn \sml_decl_text:n #1
  {
    \group_begin:
    \normalfont\small\ttfamily
    \tl_set:Nx \l_sml_decl_text_tl { \tl_to_str:n {#1} }
    \regex_replace_all:nnN { ([A-Za-z0-9_]+) }
      { \c{sml_decl_identifier:n}\cB\{\1\cE\} } \l_sml_decl_text_tl
    \regex_replace_all:nnN { ([/.:]) }
      { \1\c{allowbreak} } \l_sml_decl_text_tl
    \tl_use:N \l_sml_decl_text_tl
    \group_end:
  }
\DeclareRobustCommand{\SMLDeclText}[1]{\sml_decl_text:n{#1}}
\ExplSyntaxOff

\newcommand{\lean}[3]{\href{\SMLbase/LanglandsSecondMainLemma/#1\#L#2}{\SMLDeclText{#3}}}
\newcommand{\fmllean}[3]{\href{\FMLbase/LanglandsFirstMainLemma/#1\#L#2}{\SMLDeclText{#3}}}
\lstdefinelanguage{Lean}{
 morekeywords={theorem,lemma,def,abbrev,structure,class,instance,where,by,let,letI,in,exact,intro,noncomputable,section,end,namespace,variable,open,if,then,else},
 sensitive=true,morecomment=[l]{--},morecomment=[s]{/-}{-/},morestring=[b]"
}
\newcolumntype{P}[1]{>{\raggedright\arraybackslash}p{#1}}
\makeatletter
\expandafter\def\csname sml@D:EQ:origin-identities\endcsname{11.3}
\expandafter\def\csname smlkind@D:EQ:origin-identities\endcsname{Lemma}
\expandafter\def\csname sml@D:MX:origin-identities\endcsname{13.3}
\expandafter\def\csname smlkind@D:MX:origin-identities\endcsname{Lemma}
\expandafter\def\csname sml@D:NM:origin-identities\endcsname{12.3}
\expandafter\def\csname smlkind@D:NM:origin-identities\endcsname{Lemma}
\expandafter\def\csname sml@D:UR:norm-differences\endcsname{10.7}
\expandafter\def\csname smlkind@D:UR:norm-differences\endcsname{Lemma}
\expandafter\def\csname sml@O:I:stable-comparison\endcsname{6.14}
\expandafter\def\csname smlkind@O:I:stable-comparison\endcsname{Proposition}
\expandafter\def\csname sml@O:P:coefficients\endcsname{6.2}
\expandafter\def\csname smlkind@O:P:coefficients\endcsname{Lemma}
\expandafter\def\csname sml@U:quadratic-conversion\endcsname{4.5}
\expandafter\def\csname smlkind@U:quadratic-conversion\endcsname{Lemma}
\expandafter\def\csname sml@D:EQ:ASsymbol\endcsname{11.2}
\expandafter\def\csname smlkind@D:EQ:ASsymbol\endcsname{Lemma}
\expandafter\def\csname sml@D:EQ:FML-free-alternative\endcsname{11.14}
\expandafter\def\csname smlkind@D:EQ:FML-free-alternative\endcsname{Remark}
\expandafter\def\csname sml@D:EQ:Rcompat\endcsname{11.6}
\expandafter\def\csname smlkind@D:EQ:Rcompat\endcsname{Lemma}
\expandafter\def\csname sml@D:EQ:commmodel\endcsname{11.5}
\expandafter\def\csname smlkind@D:EQ:commmodel\endcsname{Lemma}
\expandafter\def\csname sml@D:EQ:core-exists\endcsname{11.7}
\expandafter\def\csname smlkind@D:EQ:core-exists\endcsname{Proposition}
\expandafter\def\csname sml@D:EQ:highcomparison\endcsname{11.13}
\expandafter\def\csname smlkind@D:EQ:highcomparison\endcsname{Proposition}
\expandafter\def\csname sml@D:EQ:intermediate\endcsname{11.12}
\expandafter\def\csname smlkind@D:EQ:intermediate\endcsname{Proposition}
\expandafter\def\csname sml@D:EQ:lower\endcsname{11.4}
\expandafter\def\csname smlkind@D:EQ:lower\endcsname{Lemma}
\expandafter\def\csname sml@D:EQ:main\endcsname{11.1}
\expandafter\def\csname smlkind@D:EQ:main\endcsname{Theorem}
\expandafter\def\csname sml@D:EQ:min-origin\endcsname{11.10}
\expandafter\def\csname smlkind@D:EQ:min-origin\endcsname{Proposition}
\expandafter\def\csname sml@D:EQ:min-two\endcsname{11.11}
\expandafter\def\csname smlkind@D:EQ:min-two\endcsname{Proposition}
\expandafter\def\csname sml@D:EQ:realize\endcsname{11.8}
\expandafter\def\csname smlkind@D:EQ:realize\endcsname{Proposition}
\expandafter\def\csname sml@D:MX:Rcompat\endcsname{13.5}
\expandafter\def\csname smlkind@D:MX:Rcompat\endcsname{Lemma}
\expandafter\def\csname sml@D:MX:align\endcsname{13.2}
\expandafter\def\csname smlkind@D:MX:align\endcsname{Lemma}
\expandafter\def\csname sml@D:MX:cores\endcsname{13.6}
\expandafter\def\csname smlkind@D:MX:cores\endcsname{Proposition}
\expandafter\def\csname sml@D:MX:error-estimates\endcsname{13.9}
\expandafter\def\csname smlkind@D:MX:error-estimates\endcsname{Lemma}
\expandafter\def\csname sml@D:MX:freedom\endcsname{13.8}
\expandafter\def\csname smlkind@D:MX:freedom\endcsname{Lemma}
\expandafter\def\csname sml@D:MX:higher\endcsname{13.13}
\expandafter\def\csname smlkind@D:MX:higher\endcsname{Proposition}
\expandafter\def\csname sml@D:MX:higher-coeff\endcsname{13.12}
\expandafter\def\csname smlkind@D:MX:higher-coeff\endcsname{Proposition}
\expandafter\def\csname sml@D:MX:lower\endcsname{13.4}
\expandafter\def\csname smlkind@D:MX:lower\endcsname{Lemma}
\expandafter\def\csname sml@D:MX:main\endcsname{13.1}
\expandafter\def\csname smlkind@D:MX:main\endcsname{Theorem}
\expandafter\def\csname sml@D:MX:minimal\endcsname{13.11}
\expandafter\def\csname smlkind@D:MX:minimal\endcsname{Proposition}
\expandafter\def\csname sml@D:MX:minimal-origin\endcsname{13.10}
\expandafter\def\csname smlkind@D:MX:minimal-origin\endcsname{Proposition}
\expandafter\def\csname sml@D:MX:realize\endcsname{13.7}
\expandafter\def\csname smlkind@D:MX:realize\endcsname{Proposition}
\expandafter\def\csname sml@D:MX:unit-symbol\endcsname{D.8}
\expandafter\def\csname smlkind@D:MX:unit-symbol\endcsname{Lemma}
\expandafter\def\csname sml@D:NM:alignment\endcsname{12.2}
\expandafter\def\csname smlkind@D:NM:alignment\endcsname{Lemma}
\expandafter\def\csname sml@D:NM:Qphase\endcsname{110}
\expandafter\def\csname sml@D:NM:bounds\endcsname{12.10}
\expandafter\def\csname smlkind@D:NM:bounds\endcsname{Lemma}
\expandafter\def\csname sml@D:NM:commmodel\endcsname{12.5}
\expandafter\def\csname smlkind@D:NM:commmodel\endcsname{Lemma}
\expandafter\def\csname sml@D:NM:compat\endcsname{12.6}
\expandafter\def\csname smlkind@D:NM:compat\endcsname{Lemma}
\expandafter\def\csname sml@D:NM:core-exists\endcsname{12.7}
\expandafter\def\csname smlkind@D:NM:core-exists\endcsname{Proposition}
\expandafter\def\csname sml@D:NM:freedom\endcsname{12.9}
\expandafter\def\csname smlkind@D:NM:freedom\endcsname{Lemma}
\expandafter\def\csname sml@D:NM:higher\endcsname{12.13}
\expandafter\def\csname smlkind@D:NM:higher\endcsname{Proposition}
\expandafter\def\csname sml@D:NM:lower\endcsname{12.4}
\expandafter\def\csname smlkind@D:NM:lower\endcsname{Lemma}
\expandafter\def\csname sml@D:NM:main\endcsname{12.1}
\expandafter\def\csname smlkind@D:NM:main\endcsname{Theorem}
\expandafter\def\csname sml@D:NM:minimal-origin\endcsname{12.11}
\expandafter\def\csname smlkind@D:NM:minimal-origin\endcsname{Proposition}
\expandafter\def\csname sml@D:NM:minimal-two\endcsname{12.12}
\expandafter\def\csname smlkind@D:NM:minimal-two\endcsname{Proposition}
\expandafter\def\csname sml@D:NM:realize\endcsname{12.8}
\expandafter\def\csname smlkind@D:NM:realize\endcsname{Proposition}
\expandafter\def\csname sml@D:UR:FML-free-alternative\endcsname{10.16}
\expandafter\def\csname smlkind@D:UR:FML-free-alternative\endcsname{Remark}
\expandafter\def\csname sml@D:UR:URall\endcsname{10.1}
\expandafter\def\csname smlkind@D:UR:URall\endcsname{Theorem}
\expandafter\def\csname sml@D:UR:actual\endcsname{10.4}
\expandafter\def\csname smlkind@D:UR:actual\endcsname{Proposition}
\expandafter\def\csname sml@D:UR:choosex\endcsname{10.5}
\expandafter\def\csname smlkind@D:UR:choosex\endcsname{Lemma}
\expandafter\def\csname sml@D:UR:coeffs\endcsname{10.6}
\expandafter\def\csname smlkind@D:UR:coeffs\endcsname{Proposition}
\expandafter\def\csname sml@D:UR:correctedreps\endcsname{10.10}
\expandafter\def\csname smlkind@D:UR:correctedreps\endcsname{Proposition}
\expandafter\def\csname sml@D:UR:criticalratio\endcsname{10.11}
\expandafter\def\csname smlkind@D:UR:criticalratio\endcsname{Proposition}
\expandafter\def\csname sml@D:UR:highcov\endcsname{10.14}
\expandafter\def\csname smlkind@D:UR:highcov\endcsname{Lemma}
\expandafter\def\csname sml@D:UR:minimal\endcsname{10.8}
\expandafter\def\csname smlkind@D:UR:minimal\endcsname{Proposition}
\expandafter\def\csname sml@D:UR:models\endcsname{10.3}
\expandafter\def\csname smlkind@D:UR:models\endcsname{Proposition}
\expandafter\def\csname sml@D:UR:oddT\endcsname{10.2}
\expandafter\def\csname smlkind@D:UR:oddT\endcsname{Lemma}
\expandafter\def\csname sml@D:UR:resE\endcsname{10.13}
\expandafter\def\csname smlkind@D:UR:resE\endcsname{Proposition}
\expandafter\def\csname sml@D:UR:resU\endcsname{10.12}
\expandafter\def\csname smlkind@D:UR:resU\endcsname{Proposition}
\expandafter\def\csname sml@D:UR:stable\endcsname{10.15}
\expandafter\def\csname smlkind@D:UR:stable\endcsname{Proposition}
\expandafter\def\csname sml@D:UR:traceparity\endcsname{10.9}
\expandafter\def\csname smlkind@D:UR:traceparity\endcsname{Lemma}
\expandafter\def\csname sml@D:crossed-norm\endcsname{D.5}
\expandafter\def\csname smlkind@D:crossed-norm\endcsname{Lemma}
\expandafter\def\csname sml@D:norm-separation\endcsname{D.4}
\expandafter\def\csname smlkind@D:norm-separation\endcsname{Proposition}
\expandafter\def\csname sml@D:quadratic-norms\endcsname{D.7}
\expandafter\def\csname smlkind@D:quadratic-norms\endcsname{Lemma}
\expandafter\def\csname sml@O:A:best-approximation\endcsname{C.4}
\expandafter\def\csname smlkind@O:A:best-approximation\endcsname{Lemma}
\expandafter\def\csname sml@O:A:lem:H17\endcsname{7.5}
\expandafter\def\csname smlkind@O:A:lem:H17\endcsname{Lemma}
\expandafter\def\csname sml@O:A:lem:coeff\endcsname{7.3}
\expandafter\def\csname smlkind@O:A:lem:coeff\endcsname{Lemma}
\expandafter\def\csname sml@O:A:lem:power\endcsname{3.1}
\expandafter\def\csname smlkind@O:A:lem:power\endcsname{Lemma}
\expandafter\def\csname sml@O:A:lem:twisted\endcsname{7.4}
\expandafter\def\csname smlkind@O:A:lem:twisted\endcsname{Lemma}
\expandafter\def\csname sml@O:A:prop:AS\endcsname{7.1}
\expandafter\def\csname smlkind@O:A:prop:AS\endcsname{Proposition}
\expandafter\def\csname sml@O:A:prop:g\endcsname{7.11}
\expandafter\def\csname smlkind@O:A:prop:g\endcsname{Proposition}
\expandafter\def\csname sml@O:A:thm:H14\endcsname{7.6}
\expandafter\def\csname smlkind@O:A:thm:H14\endcsname{Proposition}
\expandafter\def\csname sml@O:A:thm:h\endcsname{7.9}
\expandafter\def\csname smlkind@O:A:thm:h\endcsname{Proposition}
\expandafter\def\csname sml@O:D:quadratic\endcsname{9.12}
\expandafter\def\csname smlkind@O:D:quadratic\endcsname{Proposition}
\expandafter\def\csname sml@O:E:main\endcsname{9.25}
\expandafter\def\csname smlkind@O:E:main\endcsname{Proposition}
\expandafter\def\csname sml@O:E:trace\endcsname{9.23}
\expandafter\def\csname smlkind@O:E:trace\endcsname{Lemma}
\expandafter\def\csname sml@O:F:finite\endcsname{5.1}
\expandafter\def\csname smlkind@O:F:finite\endcsname{Proposition}
\expandafter\def\csname sml@O:F:tame\endcsname{5.3}
\expandafter\def\csname smlkind@O:F:tame\endcsname{Theorem}
\expandafter\def\csname sml@O:G:deep\endcsname{9.1}
\expandafter\def\csname smlkind@O:G:deep\endcsname{Lemma}
\expandafter\def\csname sml@O:G:lower\endcsname{8.6}
\expandafter\def\csname smlkind@O:G:lower\endcsname{Lemma}
\expandafter\def\csname sml@O:G:minimal\endcsname{9.3}
\expandafter\def\csname smlkind@O:G:minimal\endcsname{Proposition}
\expandafter\def\csname sml@O:G:reduce\endcsname{9.2}
\expandafter\def\csname smlkind@O:G:reduce\endcsname{Proposition}
\expandafter\def\csname sml@O:G:totalbranch\endcsname{9.27}
\expandafter\def\csname smlkind@O:G:totalbranch\endcsname{Theorem}
\expandafter\def\csname sml@O:I:actualcoeff\endcsname{6.9}
\expandafter\def\csname smlkind@O:I:actualcoeff\endcsname{Proposition}
\expandafter\def\csname sml@O:I:actualtwist\endcsname{6.7}
\expandafter\def\csname smlkind@O:I:actualtwist\endcsname{Proposition}
\expandafter\def\csname sml@O:I:canonical\endcsname{6.6}
\expandafter\def\csname smlkind@O:I:canonical\endcsname{Proposition}
\expandafter\def\csname sml@O:I:chooseA\endcsname{6.8}
\expandafter\def\csname smlkind@O:I:chooseA\endcsname{Lemma}
\expandafter\def\csname sml@O:I:conversion\endcsname{6.4}
\expandafter\def\csname smlkind@O:I:conversion\endcsname{Corollary}
\expandafter\def\csname sml@O:I:cubicidentity\endcsname{6.13}
\expandafter\def\csname smlkind@O:I:cubicidentity\endcsname{Proposition}
\expandafter\def\csname sml@O:I:enhanced\endcsname{6.5}
\expandafter\def\csname smlkind@O:I:enhanced\endcsname{Proposition}
\expandafter\def\csname sml@O:I:logchart\endcsname{C.3}
\expandafter\def\csname smlkind@O:I:logchart\endcsname{Lemma}
\expandafter\def\csname sml@O:I:main\endcsname{6.1}
\expandafter\def\csname smlkind@O:I:main\endcsname{Theorem}
\expandafter\def\csname sml@O:I:normdepths\endcsname{6.10}
\expandafter\def\csname smlkind@O:I:normdepths\endcsname{Lemma}
\expandafter\def\csname sml@O:I:normlog\endcsname{6.3}
\expandafter\def\csname smlkind@O:I:normlog\endcsname{Lemma}
\expandafter\def\csname sml@O:I:Sdepth\endcsname{6.12}
\expandafter\def\csname smlkind@O:I:Sdepth\endcsname{Lemma}
\expandafter\def\csname sml@O:I:scalarphase\endcsname{6.11}
\expandafter\def\csname smlkind@O:I:scalarphase\endcsname{Proposition}
\expandafter\def\csname sml@O:M:monophase\endcsname{8.2}
\expandafter\def\csname smlkind@O:M:monophase\endcsname{Lemma}
\expandafter\def\csname sml@O:M:polycoeff\endcsname{8.3}
\expandafter\def\csname smlkind@O:M:polycoeff\endcsname{Lemma}
\expandafter\def\csname sml@O:M:rigidity\endcsname{A.5}
\expandafter\def\csname smlkind@O:M:rigidity\endcsname{Lemma}
\expandafter\def\csname sml@O:M:compat\endcsname{8.4}
\expandafter\def\csname smlkind@O:M:compat\endcsname{Lemma}
\expandafter\def\csname sml@O:M:global\endcsname{8.5}
\expandafter\def\csname smlkind@O:M:global\endcsname{Proposition}
\expandafter\def\csname sml@O:M:lambdaformula\endcsname{8.7}
\expandafter\def\csname smlkind@O:M:lambdaformula\endcsname{Lemma}
\expandafter\def\csname sml@O:M:preimage-lemma\endcsname{8.1}
\expandafter\def\csname smlkind@O:M:preimage-lemma\endcsname{Lemma}
\expandafter\def\csname sml@O:M:realize\endcsname{8.8}
\expandafter\def\csname smlkind@O:M:realize\endcsname{Proposition}
\expandafter\def\csname sml@O:N:Q2\endcsname{9.15}
\expandafter\def\csname smlkind@O:N:Q2\endcsname{Proposition}
\expandafter\def\csname sml@O:N:Z\endcsname{9.14}
\expandafter\def\csname smlkind@O:N:Z\endcsname{Proposition}
\expandafter\def\csname sml@O:N:moments\endcsname{9.13}
\expandafter\def\csname smlkind@O:N:moments\endcsname{Lemma}
\expandafter\def\csname sml@O:P:choose\endcsname{8.14}
\expandafter\def\csname smlkind@O:P:choose\endcsname{Lemma}
\expandafter\def\csname sml@O:P:four\endcsname{8.17}
\expandafter\def\csname smlkind@O:P:four\endcsname{Proposition}
\expandafter\def\csname sml@O:P:pullback\endcsname{8.15}
\expandafter\def\csname smlkind@O:P:pullback\endcsname{Proposition}
\expandafter\def\csname sml@O:P:twisted-coefficients\endcsname{8.16}
\expandafter\def\csname smlkind@O:P:twisted-coefficients\endcsname{Proposition}
\expandafter\def\csname sml@O:Q:logentry\endcsname{9.4}
\expandafter\def\csname smlkind@O:Q:logentry\endcsname{Proposition}
\expandafter\def\csname sml@O:R:R2\endcsname{8.13}
\expandafter\def\csname smlkind@O:R:R2\endcsname{Proposition}
\expandafter\def\csname sml@O:R:comm-log\endcsname{8.9}
\expandafter\def\csname smlkind@O:R:comm-log\endcsname{Lemma}
\expandafter\def\csname sml@O:R:normalize\endcsname{8.12}
\expandafter\def\csname smlkind@O:R:normalize\endcsname{Proposition}
\expandafter\def\csname sml@O:R:shift\endcsname{8.10}
\expandafter\def\csname smlkind@O:R:shift\endcsname{Lemma}
\expandafter\def\csname sml@O:R:T\endcsname{8.11}
\expandafter\def\csname smlkind@O:R:T\endcsname{Lemma}
\expandafter\def\csname sml@O:T:Q124\endcsname{9.21}
\expandafter\def\csname smlkind@O:T:Q124\endcsname{Proposition}
\expandafter\def\csname sml@O:T:Q3\endcsname{9.22}
\expandafter\def\csname smlkind@O:T:Q3\endcsname{Proposition}
\expandafter\def\csname sml@O:T:actualnorm\endcsname{9.20}
\expandafter\def\csname smlkind@O:T:actualnorm\endcsname{Lemma}
\expandafter\def\csname sml@O:T:rational\endcsname{9.16}
\expandafter\def\csname smlkind@O:T:rational\endcsname{Lemma}
\expandafter\def\csname sml@O:T:representatives\endcsname{9.19}
\expandafter\def\csname smlkind@O:T:representatives\endcsname{Lemma}
\expandafter\def\csname sml@O:W:actualnorm\endcsname{9.5}
\expandafter\def\csname smlkind@O:W:actualnorm\endcsname{Lemma}
\expandafter\def\csname sml@O:W:final\endcsname{9.9}
\expandafter\def\csname smlkind@O:W:final\endcsname{Proposition}
\expandafter\def\csname sml@O:W:split\endcsname{9.6}
\expandafter\def\csname smlkind@O:W:split\endcsname{Lemma}
\expandafter\def\csname sml@U:common-function\endcsname{4.8}
\expandafter\def\csname smlkind@U:common-function\endcsname{Lemma}
\expandafter\def\csname sml@U:common-origin\endcsname{4.7}
\expandafter\def\csname smlkind@U:common-origin\endcsname{Lemma}
\expandafter\def\csname sml@U:conjugacy\endcsname{4.2}
\expandafter\def\csname smlkind@U:conjugacy\endcsname{Proposition}
\expandafter\def\csname sml@U:determinant\endcsname{4.2}
\expandafter\def\csname smlkind@U:determinant\endcsname{Proposition}
\expandafter\def\csname sml@U:dispatch\endcsname{14.1}
\expandafter\def\csname smlkind@U:dispatch\endcsname{Lemma}
\expandafter\def\csname sml@U:e2\endcsname{4.6}
\expandafter\def\csname smlkind@U:e2\endcsname{Lemma}
\expandafter\def\csname sml@U:main\endcsname{1.1}
\expandafter\def\csname smlkind@U:main\endcsname{Theorem}
\expandafter\def\csname sml@U:missing-coset\endcsname{4.9}
\expandafter\def\csname smlkind@U:missing-coset\endcsname{Proposition}
\expandafter\def\csname sml@U:odd-power\endcsname{4.3}
\expandafter\def\csname smlkind@U:odd-power\endcsname{Corollary}
\expandafter\def\csname sml@U:quadratic-square\endcsname{4.4}
\expandafter\def\csname smlkind@U:quadratic-square\endcsname{Proposition}
\expandafter\def\csname sml@U:stationary-factor\endcsname{4.1}
\expandafter\def\csname smlkind@U:stationary-factor\endcsname{Lemma}
\expandafter\def\csname sml@U:tame-covector\endcsname{5.2}
\expandafter\def\csname smlkind@U:tame-covector\endcsname{Lemma}
\expandafter\def\csname sml@sec:ramification\endcsname{3}
\expandafter\def\csname smlkind@sec:ramification\endcsname{Section}
\expandafter\def\csname sml@D:EQ:bounds\endcsname{11.9}
\expandafter\def\csname smlkind@D:EQ:bounds\endcsname{Lemma}
\expandafter\def\csname sml@D:EQ:pullback-adjust\endcsname{94}
\newcommand{\SMLnum}[1]{\@ifundefined{sml@#1}{\PackageError{SMLFormal}{Unknown companion label #1}{Refresh the companion reference map.}}{\csname sml@#1\endcsname}}
\newcommand{\SMLkind}[1]{\csname smlkind@#1\endcsname}
\DeclareRobustCommand{\SMcite}[1]{\cite[\SMLkind{#1}~\SMLnum{#1}]{SML}}
\makeatother

\title[Formalization of Langlands's Second Main Lemma]
{Formalization of Langlands's Second Main Lemma\\for Local Epsilon Factors}
\author{Fukuhiro Ueda}
\address{Research Institute for Mathematical Sciences, Kyoto University, Kyoto 606--8502, Japan}
\email{fueda@kurims.kyoto-u.ac.jp}
\date{October 2026}
\subjclass[2020]{11S37, 11S40, 11R42, 03B35, 68V20}
\keywords{Lean 4, formalized mathematics, local epsilon factors, local constants,
Second Main Lemma, ramification, Lamprecht's formula, finite sums}
\hypersetup{pdftitle={Formalization of Langlands's Second Main Lemma for Local Epsilon Factors},pdfauthor={Fukuhiro Ueda},bookmarksdepth=2}
\begin{document}
\begin{abstract}
We describe the Lean~4 formalization of Langlands's Second Main Lemma
for local epsilon factors over nonarchimedean local fields, in mixed and
equal characteristic. The lemma compares the induction expressions attached
to compatible characters of the intermediate fields of a
bicyclic extension of prime-square degree. 
We explain the mathematical constructions
and comparisons that constitute the proof, and identify their Lean
declarations.    
\end{abstract}
\maketitle
\tableofcontents

\section{Introduction}\label{sec:intro}
\subsection{The Second Main Lemma}\label{subsec:second-main-statement}

Let $F$ be a nonarchimedean local field of residue characteristic $p$,
and let $K/F$ be finite Galois with
$\Gal(K/F)\simeq C_\ell\times C_\ell$, where $\ell$ is prime.
For a degree-$\ell$ intermediate field $L$, put
\[
 S(L/F)=\{\mu\in\operatorname{Hom}_{\mathrm{cont}}(F^\times,\C^\times):
                    \mu\circ\Nm_{L/F}=1\}.
\]
In Lean, $S(L/F)$ is represented by \code{NormCharacter F L}.
Fix a nontrivial (continuous) additive character $\psi_F$, and write
$\psi_E=\psi_F\circ\Tr_{E/F}$ for every extension $E/F$ used below.
All multiplicative characters in the statement are continuous
quasi-characters.

Let $\Theta$ be a Galois-invariant quasi-character of $K^\times$.
For each degree-$\ell$ intermediate field $L$, let $\theta_L$ be a
quasi-character of $L^\times$ with
$\theta_L\circ\Nm_{K/L}=\Theta$. Such a family $(\theta_L)$ is called
\emph{compatible}. It is \emph{primitive} if $\Theta$ is not a norm
pullback from $F^\times$. Primitive compatible families are represented by
\lean{Basic/SecondMainStatement.lean}{81}{PrimitiveCompatibleFamily}.
For a primitive compatible family $(\theta_L)$, the Second Main
Lemma, \SMcite{U:main}, asserts that
\begin{equation}\label{eq:induction-expression}
 \mathcal A_L(\theta_L)=\Delta_L(\theta_L,\psi_L)
                \prod_{\mu\in S(L/F)}\Delta_F(\mu,\psi_F)
\end{equation}
is independent of $L$. Here $\Delta_E$ is Langlands's local constant,
recalled in Section~\ref{sec:principles}. The Lean definition
\lean{Basic/SecondMainStatement.lean}{128}{inductionExpression}
represents $\mathcal A_L(\theta_L)$. This independence is one of
the identities used in the induction construction of local epsilon
factors~\cite{Langlands}.

The inducing characters exist for every invariant $\Theta$; for a primitive
family the local constant is independent of the inducing choice. These points
are discussed in Section~\ref{subsec:compatible-powers}.

\subsection{Structure of the proof}

The mathematical proof is given in \cite{SML}.  The formalization uses Lean~4~\cite{Lean},
Mathlib~\cite{Mathlib}, and the author's First Main Lemma
development~\cite{FMLFormal}. It also records the few places where the
final mathematical proof and the Lean proof use different routes.

The finite residue extension is cyclic, so a $C_\ell^2$-extension cannot
be unramified. Its inertia has order $\ell$ or $\ell^2$. If
$\ell\ne p$, tame inertia is cyclic and only the first possibility
occurs. Section~\ref{sec:tame} compares the unique unramified
degree-$\ell$ intermediate field with every ramified one. The
conductor-one calculation is a finite-field identity; higher conductors
are reached by a common twist from $F^\times$.

For $\ell=p>2$, Section~\ref{sec:odd-ur} treats the case with an
unramified intermediate field. Section~\ref{sec:odd-total} treats total
ramification. The latter comparison requires three connected
constructions: Artin--Schreier coordinates and trace--norm estimates;
compatible characters of minimal conductor; and norm representatives
for their twists. Lamprecht's formula writes each local constant as the product of an
explicit factor involving multiplicative and additive character values
and a normalized sum over the residue field. The estimates show that
the local-field factors cancel in the comparison, leaving only the
normalized finite-field sums. Centering these quadratic sums makes their
quotient a fourth root of unity.
The First Main Lemma gives a $p$th-power relation, which finishes the
odd-degree argument. The case $p=3$ retains a cubic contribution that
must be canceled with the quadratic correction.

When $\ell=p=2$, the First Main Lemma determines the square of
\eqref{eq:induction-expression}, leaving its sign to be determined.
Sections~\ref{sec:dyadic-ur}--\ref{sec:dyadic-maximal} compare the
finite sums themselves. A recurring dyadic construction uses an auxiliary element of $K$ to
organize the stationary coefficients. This gives the model
minimal-conductor comparison; in the higher-conductor ranges the
coefficients are adjusted separately. The exact choices are given in
the relevant case sections and in the companion paper.
Varying the chosen data relates the critical functions, which encode
the residue-field part of the local constants. At equal breaks the
relevant residue maps have images of index two; the third quadratic
intermediate field supplies the distinct kernels that force
cancellation on the omitted coset.

The classical arguments in the companion's appendices are cited where
they are used and described briefly in Section~\ref{subsec:auxiliary}. The proofs do not invoke the local class field correspondence.

\subsection{The exported Lean theorem}

The proposition asserting that \eqref{eq:induction-expression} is independent of $L$ is
\lean{Basic/SecondMainStatement.lean}{169}{SecondMainIdentity}, and its proof is \lean{Main.lean}{37}{secondMainLemma}.
The exported declaration is:
\begin{lstlisting}
theorem secondMainLemma
    (F K : Type)
    [Field F] [ValuativeRel F] [TopologicalSpace F]
    [IsNonarchimedeanLocalField F]
    [Field K] [ValuativeRel K] [TopologicalSpace K]
    [IsNonarchimedeanLocalField K]
    [Algebra F K] [ValuativeExtension F K]
    [Module.Free F K] [Module.Finite F K]
    {ℓ : ℕ} (extension : Basic.PrimeSquareExtension F K ℓ)
    (family : Basic.PrimitiveCompatibleFamily F K extension)
    (ψF : ContinuousAddChar F) (hψF : ψF ≠ 1) :
    Basic.SecondMainIdentity F K extension family ψF hψF :=
  secondMainIdentity_of_actualData F K extension family ψF hψF
\end{lstlisting}
The field condition, compatible valuations, and finite extension are supplied
by the displayed typeclasses. The remaining algebraic hypotheses are packaged
in \lean{Basic/SecondMainStatement.lean}{47}{PrimeSquareExtension}:
\begin{lstlisting}
structure PrimeSquareExtension
    (F K : Type) [Field F] [Field K] [Algebra F K]
    [Module.Free F K] [Module.Finite F K] (ℓ : ℕ) : Prop where
  prime : ℓ.Prime
  isGalois : IsGalois F K
  galoisGroupEquiv : Nonempty
    (Gal(K/F) ≃* (Multiplicative (ZMod ℓ) ×
                  Multiplicative (ZMod ℓ)))
\end{lstlisting}
An intermediate field has type \codeexpr{IntermediateField F K};
\codeexpr{Module.finrank F L = \(\ell\)} is its degree condition.
The restricted valuation, topology, local-field structure, and compatible
cyclic towers are supplied by
\lean{Basic/IntermediateValuation.lean}{45}{intermediateFieldValuativeRel},
\lean{Basic/IntermediateCompleteness.lean}{42}{intermediateFieldTopology},
\lean{Basic/Fields.lean}{47}{intermediateField_localField}, and
\lean{Basic/Fields.lean}{157}{intermediateField_tower_compatible}.

The Lean
links point to the revisions listed in Section~\ref{sec:artifact}.

\section{Local constants and comparison principles}\label{sec:principles}
\subsection{Local constants, ramification, and conductors}
\label{subsec:ramification-conductors}

For a nonarchimedean local field $E$, normalize
$v_E(E^\times)=\Z$, write $\OO_E$ for its valuation ring,
$\pp_E$ for its maximal ideal, and $k_E$ for its residue field.
Put $U_E^0=\OO_E^\times$ and $U_E^r=1+\pp_E^r$ for $r\ge1$.
For a quasi-character $\theta$ of $E^\times$, its conductor
$m_E(\theta)$ is the least $m\ge0$ for which $\theta|_{U_E^m}=1$. The largest trivial fractional ideal of a
nontrivial additive character $\psi$ is $\pp_E^{-n_E(\psi)}$.

An element $\Gamma\in E^\times$ is admissible for $(\theta,\psi)$
if $v_E(\Gamma)=m_E(\theta)+n_E(\psi)$. For a positive Haar measure
on $U_E^0$, Langlands's local constant is
\[
 \Delta_E(\theta,\psi)=\theta(\Gamma)\,
 \operatorname{ph}\!\left(\int_{U_E^0}
                  \psi(u/\Gamma)\theta(u)^{-1}\,du\right),
 \qquad \operatorname{ph}(w)=w/|w|\quad(w\ne0).
\]
The integral is nonzero by \cite[Corollary~2.5]{FML}.
Independence of the admissible element and of the positive Haar
normalization is proved in \cite[Lemma~2.7]{FML}.

For $m=m_E(\theta)>0$, the integral is a positive multiple of
\[
 \sum_{u\in U_E^0/U_E^m}\psi(u/\Gamma)\theta(u)^{-1}.
\]
Thus it has the phase of a finite Gauss sum. For $m=0$ the same
definition gives
$\Delta_E(\theta,\psi)=\theta(\varpi_E)^{n_E(\psi)}$, where
$\varpi_E$ is a uniformizer of $E$.
In Lean, the local constant is
\code{LanglandsFirstMainLemma.localConstant};
\lean{Basic/LocalConstants.lean}{50}{localConstant_eq_finite}
proves its finite-sum formula.

For cyclic prime-degree $E/F$ and a quasi-character $\lambda$ of
$F^\times$, the imported First Main Lemma \cite[Theorem~1.1]{FML} is
\begin{equation}\label{eq:first-main}
 \Delta_E(\lambda\circ\Nm_{E/F},\psi_E)
       \prod_{\mu\in S(E/F)}\Delta_F(\mu,\psi_F)
   =\prod_{\mu\in S(E/F)}\Delta_F(\lambda\mu,\psi_F).
\end{equation}
For $c\in E^\times$, one also has
\[
 \Delta_E(1,\psi)=1,\qquad
 \Delta_E(\theta,\psi(c\,\cdot))=\theta(c)\Delta_E(\theta,\psi),
\]
and, for a finite-order character $\mu$,
\[
 \Delta_E(\mu,\psi)\Delta_E(\mu^{-1},\psi)=\mu(-1).
\]
These identities are \cite[Lemmas~2.10, 2.12, 2.13]{FML}.
Pairing inverse norm characters shows that the norm-character
product in \eqref{eq:first-main} is $1$ when $[E:F]$ is odd. For
quadratic $E/F$ it is $\Delta_F(\omega_{E/F},\psi_F)$, where
$\omega_{E/F}$ is the nontrivial member of $S(E/F)$.

For a finite separable extension $E/F$, let $e_{E/F}$ be its
ramification index and $D_{E/F}$ its different exponent. The trace
formula and its conductor consequence
\cite[proof of Theorem~3.1(a) and Corollary~3.7]{FML} are
\begin{equation}\label{eq:trace-conductor}
 \Tr_{E/F}(\pp_E^r)=
       \pp_F^{\lfloor(r+D_{E/F})/e_{E/F}\rfloor},
 \qquad
 n_E(\psi_E)=e_{E/F}n_F(\psi_F)+D_{E/F}.
\end{equation}
The valuations in this formula are normalized separately on their
fields. In particular,
$v_F(\Nm_{E/F}(z))=f_{E/F}v_E(z)$, where $f_{E/F}$ is the residue
degree. For a totally ramified extension the norm preserves the
integer value of the valuation, whereas restricting $v_E$ to
$F^\times$ multiplies $v_F$ by $e_{E/F}$.

For cyclic prime-degree $E/F$, the quotient
$F^\times/\Nm_{E/F}(E^\times)$ has order $[E:F]$.
If the extension is ramified with lower break $t$, each nontrivial
norm character has conductor $t+1$. Write
$\psi^{\mathrm H}_{E/F}$ for the inverse Herbrand function. In the
ramified degree-$\ell$ case,
\[
 D_{E/F}=(\ell-1)(t+1),\qquad
 \psi^{\mathrm H}_{E/F}(s)=
 \begin{cases}
 s,&0\le s\le t,\\
 t+\ell(s-t),&s\ge t.
 \end{cases}
\]
For integral $s\ge0$, the norm induces a map from
$U_E^{\psi^{\mathrm H}_{E/F}(s)}/U_E^{\psi^{\mathrm H}_{E/F}(s)+1}$
to $U_F^s/U_F^{s+1}$. It is an isomorphism for $s\ne t$;
at $s=t$ its kernel and cokernel have order $\ell$. Moreover,
$\Nm_{E/F}(U_E^{\psi^{\mathrm H}_{E/F}(s)})=U_F^s$ for integral
$s>t$. For an unramified extension,
$\Nm_{E/F}(U_E^s)=U_F^s$ for every integer $s\ge0$.
These are the norm-filtration results recalled in
\SMcite{sec:ramification}; the formal input is
\fmllean{Ramification/NormFiltration.lean}{838}{cyclicPrimeNormFiltration}.

For a ramified cyclic extension $E/F$ of prime degree and break $t$,
the \emph{norm-approximation estimate} \cite[(10)]{SML} states:
for $b\in F^\times$ with $v_F(b)=a$ and an integer $0\le s\le t$,
there is $x\in E^\times$ such that
\begin{equation}\label{eq:norm-approximation}
 v_E(x)=a,\qquad b-\Nm_{E/F}(x)\in\pp_F^{a+s}.
\end{equation}

If $n=m_F(\lambda)$ and $E/F$ is ramified cyclic of degree $\ell$,
the pullback conductor is
\begin{equation}\label{eq:pullback-conductor}
 m_E(\lambda\circ\Nm_{E/F})=
 \begin{cases}
 n,&n\le t,\\
 \ell(n-1)-(\ell-1)t+1,&n>t+1.
 \end{cases}
\end{equation}
At $n=t+1$ the pullback has conductor at most $n$; it has smaller
conductor precisely when a twist $\lambda\mu$, with $\mu\in S(E/F)$,
has smaller conductor. In an unramified extension the conductor is
unchanged.

The comparisons use a distinguished intermediate field: the unramified
field when it exists, and otherwise a field of smallest break.
Its norm subgroup differs from that of every other intermediate field
by \SMcite{D:norm-separation}, formalized by
\lean{Ramification/NormSeparation.lean}{1557}{normSeparation}.
In the biquadratic case, all three norm subgroups are distinct and
their nontrivial norm characters satisfy
$\omega_1\omega_2\omega_3=1$ (\SMcite{D:quadratic-norms});
both assertions are formalized by
\lean{Characters/QuadraticProduct.lean}{482}{quadraticProduct}.
For $L_i,L_j$ with distinct norm subgroups, norm pullback from
$F^\times$ identifies
$S(L_j/F)$ with $S(K/L_i)$, by \SMcite{D:crossed-norm} and
\lean{Characters/CrossedNorm.lean}{180}{crossedNorm}.

\subsection{Stationary coefficients and Lamprecht's formula}

Let $\theta$ be a quasi-character of $E^\times$ with conductor
$m>1$, and choose an admissible element $\Gamma$.
There is a unique class
$\beta\in\OO_E/\pp_E^{\lfloor m/2\rfloor}$ whose representatives
satisfy
\begin{equation*}
 \theta(1+z)=\psi_E(\beta z/\Gamma)
          \qquad(z\in\pp_E^{\lceil m/2\rceil}).
\end{equation*}
Every representative of $\beta$ is a unit. The stationary
class, its existence and uniqueness, and the valuation of its
representatives are formalized, respectively, by
\fmllean{Lamprecht/StationaryClass.lean}{131}{stationaryNumeratorClass},
\fmllean{Lamprecht/StationaryClass.lean}{266}{stationaryNumeratorClass_existsUnique},
and
\fmllean{Lamprecht/StationaryClass.lean}{301}{stationaryNumeratorClass_representative_ord}.

Fix $\alpha\in E^\times$, put
$\Psi_E(z)=\psi_E(\alpha z)$, and let $\pp_E^J$ be its largest trivial
ideal. In the following convention, a \emph{stationary coefficient}
$Z\in E^\times$ satisfies
\begin{equation}\label{eq:stationary-coefficient}
 v_E(Z)=J-m,\qquad
 \theta(1-z)=\Psi_E(Zz)
          \quad(z\in\pp_E^{\lceil m/2\rceil}).
\end{equation}
The element $-\alpha Z\Gamma$ represents the stationary class $\beta$.
In the case comparisons $\alpha$ belongs to
$F^\times$, so
$\Psi_E=\Psi_F\circ\Tr_{E/F}$. Formula
\eqref{eq:trace-conductor} then gives the largest trivial ideal in
each extension.

Lamprecht's formula in these conventions, \SMcite{U:stationary-factor},
is
\begin{equation}\label{eq:normalized-factor}
 \Delta_E(\theta,\psi_E)
   =\theta((-\alpha Z)^{-1})\Psi_E(-Z)g_E(\theta,Z).
\end{equation}
The first two factors are character values; $g_E(\theta,Z)$ is the
residue-field factor. For even $m$ it equals $1$. For odd $m$, choose a critical element
$\delta\in E$ of valuation $\lfloor m/2\rfloor$. For an integral
lift $x$ of $\bar x\in k_E$, put
\begin{equation}\label{eq:critical-function}
 \begin{aligned}
 H_{\theta,Z,\delta}(\bar x)
   &=\Psi_E(-Z\delta x)\theta(1+\delta x)^{-1},\\
 g_E(\theta,Z)&=|k_E|^{-1/2}
                     \sum_{x\in k_E}H_{\theta,Z,\delta}(x).
 \end{aligned}
\end{equation}
The function $H_{\theta,Z,\delta}$ is called the critical function.
It is well defined. Changing $\delta$ reindexes its sum,
so the normalized sum is independent of $\delta$. It can depend on
the chosen stationary coefficient $Z$.

The even- and odd-conductor formulas are
\fmllean{Lamprecht/Formula.lean}{2373}{lamprechtEven} and
\fmllean{Lamprecht/Formula.lean}{2403}{lamprechtOdd}. The combined formula used here is
\lean{Stationary/NormalizedFactor.lean}{327}{normalizedFactor}.

Writing
$\psi_0(\bar z)=\Psi_E(-Z\delta^2z)$ for an integral lift $z$, the
critical function has polar pairing
\begin{equation*}
 \frac{H_{\theta,Z,\delta}(x+y)}
      {H_{\theta,Z,\delta}(x)H_{\theta,Z,\delta}(y)}
       =\psi_0(xy).
\end{equation*}
The character $\psi_0$ is nontrivial, so the pairing is nondegenerate
and the sum in \eqref{eq:critical-function} has absolute value
$|k_E|^{1/2}$. The declarations
\fmllean{Lamprecht/Formula.lean}{575}{lamprechtHasseValue_map_add}
and
\fmllean{Lamprecht/Formula.lean}{277}{lamprechtResidualAddChar_ne_one}
give the polar identity and its nondegeneracy.

In odd residue characteristic, a function with the polar pairing above has the form
$\psi_0(x^2/2+b x)$ for a uniquely determined $b\in k_E$.
Translation $x\mapsto x-b$ gives
\[
 |k_E|^{-1/2}\sum_{x\in k_E}\psi_0(x^2/2+b x)
   =\psi_0(-b^2/2)\,
       |k_E|^{-1/2}\sum_{x\in k_E}\psi_0(x^2/2).
\]
The centered quadratic sum is a fourth root of unity.
The relevant declarations are
\fmllean{Lamprecht/CriticalPolarCoordinate.lean}{293}{eq_quadratic},
\fmllean{Lamprecht/CriticalPolarCoordinate.lean}{315}{affineCoefficient_unique},
\fmllean{FiniteField/QuadraticPhase.lean}{296}{quadraticPhase_eq_basic},
\fmllean{FiniteField/QuadraticPhase.lean}{307}{quadraticPhase_basic_sq},
and
\fmllean{FiniteField/QuadraticPhase.lean}{278}{quadraticPhase_completeSquare}.
The translation of general critical sums is
\fmllean{FiniteField/HasseFunction.lean}{149}{sum_translate}.

If a quasi-character $\nu$ satisfies
$m_E(\nu)\le\lfloor m/2\rfloor$, the stable-twist formula is
\begin{equation*}
 \Delta_E(\theta\nu,\psi_E)
   =\nu((-\alpha Z)^{-1})\Delta_E(\theta,\psi_E).
\end{equation*}
It is proved by
\fmllean{Lamprecht/StableTwist.lean}{78}{stationaryNumeratorClass_stableTwist}
and \fmllean{Lamprecht/StableTwist.lean}{741}{stableTwist}.
The bound makes $\nu$ trivial on the unit group used in the critical
sum. 

\subsection{Compatible characters and the power relations}
\label{subsec:compatible-powers}

For every Galois-invariant $\Theta$, the inducing characters used above exist:
Hilbert~90 allows descent to the norm image, norm filtration gives continuity
on that image, and continuous character extension gives a character of
$L^\times$. This is
\lean{Characters/InducingChoice.lean}{123}{invariantCharacter_inducing_exists}.

For distinct intermediate fields $L_i,L_j$, equality of the two norm
pullbacks makes their common character invariant under
$\Gal(K/L_i)$ and $\Gal(K/L_j)$, hence under $\Gal(K/F)$.
This implication is
\lean{Characters/Conjugacy.lean}{94}{compatible_invariant}.
Now let the compatible family be primitive, and let $\sigma_i$
generate $\Gal(L_i/F)$, with
$\theta_i^{\sigma_i}(x)=\theta_i(\sigma_i^{-1}(x))$.
Then $\theta_i^{\sigma_i}/\theta_i$ is a nontrivial member of
$S(K/L_i)$ and therefore a generator. If it were trivial,
Hilbert~90 and character extension would descend $\theta_i$ through
$\Nm_{L_i/F}$, contradicting primitivity.

The crossed-norm description of $S(K/L_i)$ shows that its members are
fixed by $\Gal(L_i/F)$. Consequently the conjugates of $\theta_i$
are precisely its twists by $S(K/L_i)$. This is
\SMcite{U:determinant}, formalized through
\lean{Characters/Conjugacy.lean}{128}{ConjugateTwistData}
and \lean{Characters/Conjugacy.lean}{518}{conjugacy}.
Since $\psi_{L_i}$ is invariant under $\sigma_i$, conjugating
$\theta_i$ preserves $\Delta_{L_i}(\theta_i,\psi_{L_i})$.
Two inducing choices on $L_i^\times$ differ by a member of $S(K/L_i)$,
hence are conjugate and have the same local constant.

There is also a character of $F^\times$ common to all intermediate
fields. With
$\epsilon_i=\prod_{\mu\in S(L_i/F)}\mu$, it is
\begin{equation}\label{eq:determinant-restriction}
 D_\theta=\theta_i|_{F^\times}\epsilon_i.
\end{equation}
For a pair with distinct norm subgroups, compatibility gives this
equality on both norm images. The two subgroups have index $\ell$
and generate $F^\times$, so the equality holds on $F^\times$.
Comparing each intermediate field with the distinguished field above
gives the common character $D_\theta$. For odd $\ell$,
$\epsilon_i=1$. For $\ell=2$, $\epsilon_i=\omega_i$.
Thus the restrictions $\theta_i|_{F^\times}$ agree in odd degree,
whereas the dyadic comparison retains a quadratic norm-character
factor.

Apply \eqref{eq:first-main} to $K/L_i$ and use the conjugate-twist
description. For odd $\ell$ the two norm-character products are $1$,
and one obtains
\begin{equation}\label{eq:power-relations}
 \Delta_{L_i}(\theta_i,\psi_{L_i})^\ell=\Delta_K(\Theta,\psi_K),
 \qquad
 \left(\frac{\mathcal A_{L_i}(\theta_i)}
              {\mathcal A_{L_j}(\theta_j)}\right)^\ell=1.
\end{equation}
This is \SMcite{U:odd-power} and
\lean{Characters/OddPower.lean}{134}{oddPower}.
In odd degree, we will  show that the ratio in
\eqref{eq:power-relations} is also a fourth root of unity, hence is one.

For a biquadratic extension, a second application of
\eqref{eq:first-main}, now to $L_i/F$ and a different norm character,
gives the common square
\begin{equation}\label{eq:quadratic-square}
 \mathcal A_{L_i}(\theta_i)^2
    =\Delta_K(\Theta,\psi_K)
                       \prod_{j=1}^3\Delta_F(\omega_j,\psi_F).
\end{equation}
It follows that any ratio of the three induction expressions is
$1$ or $-1$. This is \SMcite{U:quadratic-square} and
\lean{Characters/QuadraticSquare.lean}{126}{quadraticSquare}.
The dyadic constructions below compare the unsquared expressions by
relating their normalized critical sums.

\subsection{Common coefficients and comparison of critical functions}
\label{subsec:common-coefficients}
\label{subsec:missing-coset}

Let $K/F$ be biquadratic, with quadratic intermediate fields $L_i$
and their nontrivial norm characters $\omega_i$. Fix
$\alpha\in F^\times$ and put $\Psi_E(z)=\psi_E(\alpha z)$.
For $C\in K$, the second elementary symmetric function of its four
conjugates satisfies
\begin{equation}\label{eq:biquadratic-e2}
 e_2(C)=\Tr_{L_i/F}(\Nm_{K/L_i}(C))
                      +\Nm_{L_i/F}(\Tr_{K/L_i}(C)).
\end{equation}
This identity is
\SMcite{U:e2}, formalized by
\lean{Algebra/BiquadraticE2.lean}{126}{biquadraticE2}.

Suppose $(\theta_i)$ is primitive compatible, both
$m_i=m_{L_i}(\theta_i)$ and $m_i^\omega=m_F(\omega_i)$ exceed one,
and $C\in K^\times$ has nonzero trace to each $L_i$. Assume that
$\Nm_{K/L_i}(C)$ and
$\Nm_{L_i/F}(\Tr_{K/L_i}(C))$ are stationary coefficients for
$\theta_i$ and $\omega_i$, respectively. In particular, they have
the valuations required in \eqref{eq:stationary-coefficient} and
satisfy
\[
 \begin{aligned}
 \theta_i(1-z)&=\Psi_{L_i}(\Nm_{K/L_i}(C)z)
       &&(z\in\pp_{L_i}^{\lceil m_i/2\rceil}),\\
 \omega_i(1-z)&=\Psi_F(\Nm_{L_i/F}(\Tr_{K/L_i}(C))z)
       &&(z\in\pp_F^{\lceil m_i^\omega/2\rceil}).
 \end{aligned}
\]
Let $g_{\theta_i}$ and $g_{\omega_i}$ denote the normalized critical
sums for these specified coefficients, with value one in even
conductor. Multiplying \eqref{eq:normalized-factor} for the two
characters gives
\begin{equation}\label{eq:common-origin}
 \mathcal A_{L_i}(\theta_i)
  =D_\theta(-\alpha^{-1})\Theta(C)^{-1}
             \Psi_F(-e_2(C))\,g_{\theta_i}g_{\omega_i}.
\end{equation}
Compatibility gives the factor $\Theta(C)^{-1}$, and
\eqref{eq:determinant-restriction} gives $D_\theta(-\alpha^{-1})$.
The additive-character factors combine by
\eqref{eq:biquadratic-e2}. This is
\SMcite{U:common-origin} and
\lean{Stationary/CommonOrigin.lean}{235}{commonOrigin}.
Thus a common $C$ reduces the comparison to the products
$g_{\theta_i}g_{\omega_i}$.

To compare the critical sums, we use the following identity for
character factors associated with two elements $C,C'$. For nonzero
$C,C'\in K$ with nonzero traces to all $L_i$, set
\[
 1+z_i=\frac{\Nm_{K/L_i}(C')}{\Nm_{K/L_i}(C)},\qquad
 1+w_i=\frac{\Nm_{L_i/F}(\Tr_{K/L_i}(C'))}
                  {\Nm_{L_i/F}(\Tr_{K/L_i}(C))}.
\]
Compatibility and \eqref{eq:biquadratic-e2} give
\begin{equation*}
 \begin{aligned}
 &\Psi_{L_i}(-\Nm_{K/L_i}(C)z_i)\theta_i(1+z_i)^{-1}\cdot
 \Psi_F(-\Nm_{L_i/F}(\Tr_{K/L_i}(C))w_i)\omega_i(1+w_i)^{-1}\\
 &\hspace{2em}=\Theta(C'/C)^{-1}
                         \Psi_F(-e_2(C')+e_2(C)).
 \end{aligned}
\end{equation*}
This common product, \SMcite{U:common-function}, does not require
primitivity. Its formal declaration is
\lean{Stationary/CommonFunction.lean}{65}{commonFunctionProduct_eq}.

Suppose finite abelian groups $V_i$ have the same order and carry
functions $H_i:V_i\to\C^\times$ with $H_i(0)=1$ and nondegenerate
polar bicharacters. Let $p_i:Z\to V_i$ be homomorphisms from a finite
abelian group $Z$, each with kernel of order two and image of index
two. Assume that their kernels are pairwise distinct and that
$H_i\circ p_i=Q$ is the same function on $Z$. For at least two such
indices, \SMcite{U:missing-coset} gives
\begin{equation}\label{eq:missing-coset}
 \sum_{x\in V_i}H_i(x)=\frac12\sum_{z\in Z}Q(z).
\end{equation}
The normalized sums are consequently equal.
This is formalized by
\lean{Finite/MissingCoset.lean}{30}{missingCoset}.
In the equal-break cases of Sections~\ref{sec:dyadic-equal} and
\ref{sec:dyadic-nonmaximal}, the three quadratic intermediate fields
supply the distinct kernels.

\section{The tame case}\label{sec:tame}

Assume that $K/F$ is tame and has Galois group $C_\ell^2$, with $\ell$
different from the residue characteristic $p$. The inertia subgroup is cyclic,
so there is a unique unramified degree-$\ell$ intermediate field $U$.
Every other degree-$\ell$ intermediate field $E$ is totally tamely ramified. It suffices to compare
the expression $\mathcal A_L(\theta_L)$ for $L=U,E$.

A finite-field Gauss-sum identity proves the conductor-one case.
At higher conductor, dividing by a common norm pullback from $F^\times$
gives a pair of conductor-one characters. Stable twisting and the
determinant identity \eqref{eq:determinant-restriction} then prove the comparison. This is
\SMcite{O:F:tame}, including $\ell=2$.

\subsection{Gauss-sum and the conductor-one comparison}

Let $\theta_U,\theta_E$ be a primitive compatible pair with common
norm pullback $\Theta$, and put $m=m_U(\theta_U)$. If $\sigma$ generates
$\Gal(U/F)$, then $\theta_U^\sigma/\theta_U$ generates $S(K/U)$,
and the twists of $\theta_U$ by $S(K/U)$ are precisely its conjugates,
as explained in Section~\ref{subsec:compatible-powers}.
Hence all these twists have conductor $m$. Since $K/U$ has break zero,
the conjugate quotient has conductor one. Thus $m\ge1$, and the restriction
of $\theta_U$ to $U_U^1$ is $\sigma$-invariant.
Equation~\eqref{eq:pullback-conductor}, together with its endpoint
clause, gives
\begin{equation}\label{eq:tame-conductors}
 m_K(\Theta)=1+\ell(m-1),\qquad
 m_E(\theta_E)=1+\ell(m-1).
\end{equation}
The second equality uses that $K/E$ is unramified. 

Put
\[
 \epsilon_L=\prod_{\mu\in S(L/F)}\mu,
 \qquad
 \Lambda_L=\prod_{\mu\in S(L/F)}\Delta_F(\mu,\psi_F),
 \qquad L=U,E.
\]
The comparison to be established is
\begin{equation}\label{eq:tame-target}
 \Delta_U(\theta_U,\psi_U)\Lambda_U
 =\Delta_E(\theta_E,\psi_E)\Lambda_E.
\end{equation}
For odd $\ell$, both $\epsilon_L$ are trivial and both $\Lambda_L$
are one. For $\ell=2$, $\epsilon_L$ is the nontrivial norm character and
$\Lambda_L$ is its local constant.

Write $k=k_F=k_E$, $\kappa=k_U=k_K$, and $q=|k|$.
Cyclic tame inertia implies $\ell\mid q-1$.
For a multiplicative character $\chi$ of a finite field $k$ and a
nontrivial additive character $\psi_k$, use
\[
 \tau_k(\chi)=-\sum_{z\in k^\times}\chi(z)^{-1}\psi_k(z),
 \qquad \tau_k(1)=1.
\]
The additive character on $\kappa$ is
$\psi_\kappa=\psi_k\circ\Tr_{\kappa/k}$.
Suppose that characters $\chi_\kappa$ and $\chi_0$ of
$\kappa^\times$ and $k^\times$ satisfy
$\chi_\kappa^\ell=\chi_0\circ\Nm_{\kappa/k}$ and
$\chi_0|_{\mu_\ell}\ne1$. The finite-field assertion
\SMcite{O:F:finite} is
\begin{equation*}
 \tau_\kappa(\chi_\kappa)
 =\chi_0(\ell)\tau_k(\chi_0)
       \prod_{j=1}^{\ell-1}\tau_k(\mu^j),
\end{equation*}
where $\mu$ is any character of $k^\times$ of order $\ell$. This finite-field identity is proved by
\lean{Tame/FiniteComparison.lean}{1597}{finiteComparison},
using the leading-coefficient results
\lean{Gauss/FirstCoefficient.lean}{341}{firstCoefficient},
\lean{Gauss/JacobiProduct.lean}{586}{jacobiProduct}, and
\lean{Gauss/LeadingTerm.lean}{288}{leadingTerm}.

Suppose $m_U(\theta_U)=1$. By \eqref{eq:tame-conductors},
$m_E(\theta_E)=1$. Their restrictions to units induce
$\chi_\kappa$ and $\chi_0$ with
$\chi_\kappa^\ell=\chi_0\circ\Nm_{\kappa/k}$. Primitivity makes
$\theta_U^\sigma/\theta_U$ nontrivial; on residue characters this
is equivalent to $\chi_0|_{\mu_\ell}\ne1$. 
With residue additive characters induced by $\psi_F$,
a common admissible element  and the determinant
identity \eqref{eq:determinant-restriction} give
\[
 \frac{\Delta_U(\theta_U)\Lambda_U}
      {\Delta_E(\theta_E)\Lambda_E}
 =\frac{\operatorname{ph}(\tau_\kappa(\chi_\kappa))}
 {\chi_0(\ell)\operatorname{ph}(\tau_k(\chi_0))
        \prod_{j=1}^{\ell-1}\operatorname{ph}(\tau_k(\mu^j))}.
\]
The factor $\chi_0(\ell)$ comes from the ramified trace on
residue additive characters. 
The finite-field identity above makes the displayed quotient one.
This application is
\lean{Tame/Comparison.lean}{763}{comparison_of_conductor_one}.

\subsection{Higher conductors and a common twist}

For $m>1$, Hilbert~90 and surjectivity of the unramified
norm on principal-unit groups show that
$\theta_U|_{U_U^1}$ is the norm pullback of a character of
$U_F^1$ of conductor $m$. Extend it to $\lambda$ on $F^\times$
and define
\[
 \chi_U=\theta_U/(\lambda\circ\Nm_{U/F}),\qquad
 \chi_E=\theta_E/(\lambda\circ\Nm_{E/F}).
\]
The pair remains primitive and compatible. Since $\chi_U$ is
trivial on $U_U^1$ and has nontrivial conjugate quotient, its
conductor is one; \eqref{eq:tame-conductors} gives conductor one
for $\chi_E$ as well. The construction is
\lean{Tame/ConductorReduction.lean}{333}{tame_commonTwist}.

Choose a stationary coefficient $b_F$ for $\lambda$ so that $\lambda(1+z)=\psi_F(b_Fz)$ on
$\pp_F^{\ceil{m/2}}$. By \SMcite{U:tame-covector}, this coefficient
also works for the two norm pullbacks: the domains are
$\pp_U^{\ceil{m/2}}$ and
$\pp_E^{\ceil{(1+\ell(m-1))/2}}$.
This is \lean{Tame/Covector.lean}{257}{covector}.
Stable twisting is available because $\chi_U,\chi_E$ and all members
of $S(L/F)$ have conductor at most one. It yields
\begin{equation*}
 \Delta_L(\theta_L)\Lambda_L
 =\chi_L(b_F^{-1})\epsilon_L(b_F^{-1})\Delta_F(\lambda)^\ell,
 \qquad L=U,E.
\end{equation*}
Here the First Main Lemma replaces
$\Delta_L(\lambda\circ\Nm_{L/F})\Lambda_L$ by
$\prod_{\mu\in S(L/F)}\Delta_F(\mu\lambda)$, and stable twisting
over $F$ evaluates that product. By \eqref{eq:determinant-restriction},
applied to $(\chi_U,\chi_E)$,
\[
 \chi_U(b_F^{-1})\epsilon_U(b_F^{-1})
 =\chi_E(b_F^{-1})\epsilon_E(b_F^{-1}).
\]
Hence $\Delta_L(\theta_L)\Lambda_L$ is independent of $L$, proving
\eqref{eq:tame-target} at every higher conductor.
The higher-conductor argument is
\lean{Tame/Comparison.lean}{1176}{comparison_of_conductor_gt_one};
\lean{Tame/Comparison.lean}{1472}{comparison}
combines it with the conductor-one case.

\section{Wild odd degree with an unramified intermediate field}
\label{sec:odd-ur}

Let the residue characteristic be $p>2$, and suppose that the inertia
subgroup of $\Gal(K/F)\simeq C_p^2$ has order $p$. Write $U$ for
the unramified degree-$p$ intermediate field and $E$ for any ramified
degree-$p$ intermediate field. The extensions $E/F$ and $K/U$ have
the same break $t\ge1$. Put $T=t+1$.
For a primitive compatible pair $(\theta_U,\theta_E)$, the desired
identity is $\Delta_U(\theta_U)=\Delta_E(\theta_E)$, since the two
norm-character products are one.

The characters are common twists of a compatible pair of
conductor $T$. Before stable twisting applies, we compare
the stationary coefficients with norms of a single element of $K$.
The discrepancy produces a quadratic correction in
\eqref{eq:odd-centered-change}. At the boundary $h=t$, its
character value must be combined with the trace term, especially
when $p=3$. This proves \SMcite{O:I:main}.

\subsection{Truncated logarithms and the norm character}
\label{sec:odd-logarithms}

For a local field $L$ of residue characteristic $p>2$, put
$P(Z)=\sum_{j=1}^{p-1}\frac{Z^j}{j}$.
 If $1\le r\le M$ and $pr\ge M$, then
$1-z\mapsto P(z)$ identifies $U_L^r/U_L^M$ with
$\pp_L^r/\pp_L^M$. On $\pp_L^r$, the polynomial $P$ is bijective
and preserves valuations. These assertions are
\SMcite{O:I:logchart}, formalized by
\lean{Odd/TruncatedLog.lean}{1050}{truncatedLogQuotientMulEquiv},
\lean{Odd/TruncatedLog.lean}{505}{truncatedLogOnLattice_bijective},
and \lean{Odd/TruncatedLog.lean}{198}{ord_truncatedLog}.

For a totally ramified cyclic degree-$p$ extension $L/M$ of break
$t$ and an integer $h\ge0$, \SMcite{O:I:normlog} gives
\begin{equation}\label{eq:odd-log-norm}
 P(1-\Nm_{L/M}(1-z))
 \equiv \Tr_{L/M}(P(z))+\Nm_{L/M}(P(z))
       \pmod{\pp_M^{T+h}},
 \qquad z\in\pp_L^{h+\ceil{T/p}}.
\end{equation}
For an unramified degree-$p$ extension $L/M$, the same lemma gives
\[
 P(1-\Nm_{L/M}(1-z))\equiv\Tr_{L/M}(P(z))
 \pmod{\pp_M^T},\qquad z\in\pp_L^{\ceil{T/p}}.
\]
The ramified and unramified congruences are formalized by
\lean{Odd/NormLog.lean}{1557}{normLog} and
\lean{Odd/NormLog.lean}{1596}{unramifiedNormLog}, respectively.

Let $\phi$ be arithmetic Frobenius on $U/F$. By
Sections~\ref{subsec:ramification-conductors} and~\ref{subsec:compatible-powers},
primitivity and crossed norm characters determine a nontrivial
$\tau_F\in S(E/F)$ by
$\theta_U^\phi/\theta_U=\tau_F\circ\Nm_{U/F}$. Its conductor is $T$
by the norm-character conductor statement in Section~\ref{sec:principles}.
The truncated-logarithm coefficient lemma \SMcite{O:P:coefficients}
gives $\alpha\in F^\times$ with
\begin{equation}\label{eq:odd-ur-norm-chart}
 \tau_F(1-z)=\psi_F(\alpha P(z)),
 \quad z\in\pp_F^{\ceil{T/p}},
 \qquad v_F(\alpha)=-n_F(\psi_F)-T.
\end{equation}
For $L=F,U,E,K$, put $\Psi_L(z)=\psi_L(\alpha z)$.
The largest trivial ideal of each $\Psi_L$ is $\pp_L^T$; the
different exponent $(p-1)T$ on the ramified extensions is what makes
these four exponents agree. The normalization is constructed in
\lean{Odd/UR/Setup.lean}{480}{oddURSetup_exists}.

Since $\tau_F$ is trivial on norms from $E$, applying
\eqref{eq:odd-log-norm} to \eqref{eq:odd-ur-norm-chart} and using
the bijectivity of $P$ gives, for $L/M=E/F$ or $K/U$,
\begin{equation}\label{eq:odd-norm-conversion}
 \Psi_M(\Tr_{L/M}(z)+\Nm_{L/M}(z))=1,
 \qquad
 \Psi_L(z)=\Psi_M(-\Nm_{L/M}(z)),
 \qquad z\in\pp_L^{\ceil{T/p}}.
\end{equation}
We call \eqref{eq:odd-norm-conversion} the
\emph{norm-conversion formula}. The result is \SMcite{O:I:conversion},
formalized by \lean{Odd/NormPhase.lean}{48}{normPhase}.

\subsection{Model characters and  a common twist}
\label{subsec:odd-ur-models}

Let $k=k_F=k_E$, $\kappa=k_U=k_K$, and $q=|k|$.
Choose $\bar d\in\kappa$ with $\bar d^q-\bar d=1$ and lift
$\bar c=\bar d^p-\bar d\in k$ to $c\in\OO_F$.
The polynomial $X^p-X-c$ gives an unramified generator
$d\in U$ satisfying $d^p-d=c$. In characteristic $p$,
$\phi(d)=d+1$; in mixed characteristic,
$\phi(d)-d-1\in p\OO_U$.

With the normalization \eqref{eq:odd-ur-norm-chart}, the model
construction \SMcite{O:I:canonical} produces characters
$\chi_U^0$ and $\chi_E^0$ of conductor $T$ whose norm pullbacks
are equal and primitive, with
\begin{equation*}
 \begin{aligned}
  (\chi_U^0)^\phi/\chi_U^0&=\tau_F\circ\Nm_{U/F},\\
  \chi_U^0(1-z)&=\Psi_U(d^pP(z)),
       &&z\in\pp_U^{\ceil{T/p}},\\
  \chi_E^0(1-z)&=\Psi_E(cP(z)),
       &&z\in\pp_E^{\ceil{T/p}}.
 \end{aligned}
\end{equation*}
The prescribed character on $U_U^{\ceil{T/p}}$ is defined by
$1-z\mapsto\Psi_U(d^pP(z))$. Extend it to a character $\chi_U^0$ of
$U^\times$ satisfying the first equation.
The character $\chi_U^0\circ\Nm_{K/U}$ is $\Gal(K/F)$-invariant.
By Section~\ref{subsec:compatible-powers}, choose a character $\chi_E^0$
of $E^\times$ such that
\[
 \chi_E^0\circ\Nm_{K/E}=\chi_U^0\circ\Nm_{K/U}.
\]
To obtain the principal-unit formula for $\chi_E^0$,
apply \eqref{eq:odd-log-norm} and then \eqref{eq:odd-norm-conversion}
to $dP(z)$.
The formal construction is
\lean{Odd/UR/Models.lean}{1420}{models}.

For the given pair, $\theta_U/\chi_U^0$ is Frobenius-invariant,
so it is $\lambda\circ\Nm_{U/F}$ for a quasi-character $\lambda$ of
$F^\times$. Define
$\mu_E=\theta_E/[\chi_E^0(\lambda\circ\Nm_{E/F})]$.
Compatibility puts $\mu_E$ in $S(K/E)$; it is unramified.
With $\chi_E=\chi_E^0\mu_E$, both $\chi_U^0$ and $\chi_E$ have
conductor $T$, and their displayed evaluations on principal units remain
valid. The given pair satisfies
\begin{equation}\label{eq:odd-ur-conductors}
 \begin{gathered}
  \begin{aligned}
   \theta_U&=\chi_U^0(\lambda\circ\Nm_{U/F}),&
   m_U(\theta_U)&=T+h,\\
   \theta_E&=\chi_E(\lambda\circ\Nm_{E/F}),&
   m_E(\theta_E)&=T+ph,
  \end{aligned}\\
  h\ge0,\qquad
  \begin{cases}
   m_F(\lambda)=T+h,&h>0,\\
   m_F(\lambda)\le T,&h=0.
  \end{cases}
 \end{gathered}
\end{equation}
This is \SMcite{O:I:actualtwist}, formalized in
\lean{Odd/UR/Twist.lean}{383}{twist}.

\subsection{Stationary coefficients and corrections}
\label{sec:odd-centered}

For a character $\theta$ of conductor $m\ge2$ and a scaled additive
character $\Psi_L=\psi_L(\alpha\,\cdot)$ with largest trivial
ideal $\pp_L^J$, \SMcite{O:I:enhanced} chooses $B$ such that
\[
 v_L(B)=J-m,\qquad
 \theta(1-z)=\Psi_L(BP(z)),\qquad
 z\in\pp_L^{\floor{m/2}}.
\]
When $m=2d+1$, a critical element $\delta$ of valuation $d$ gives
the centered critical function
$\bar z\mapsto\Psi_L(-B\delta^2z^2/2)$, where $z$ lifts
$\bar z$. Its normalized sum, denoted here by $g_L(\theta,B)$,
is in $\mu_4$. For even $m$ this factor is one.

If a norm construction supplies $Z=B+\eta$ instead, with
$v_L(\eta)\ge J-\ceil{m/2}$, the local constant is
\begin{equation}\label{eq:odd-centered-change}
 \Delta_L(\theta,\psi_L)
 =\theta((-\alpha Z)^{-1})\Psi_L(-Z)
   \Psi_L\!\left(\frac{\eta^2}{2B}\right)g_L(\theta,B).
\end{equation}
The coefficient-change formula and these bounds are proved by
\lean{Odd/EnhancedStationary.lean}{973}{enhancedStationary}.

First assume $0\le h\le t$ in \eqref{eq:odd-ur-conductors}, and
put $d_U=\floor{(T+h)/2}$, $d_E=\floor{(T+ph)/2}$.
For the common twisting quasi-character $\lambda$ constructed above, the logarithm
quotient supplies a coefficient $\gamma\in F$ on $\pp_F^{d_U}$:
$\lambda(1-z)=\psi_F(\gamma P(z))$.
For $h>0$, $v_F(\gamma/\alpha)=-h$; for $h=0$, choose
$\gamma/\alpha\in\OO_F$.

Applying norm approximation \eqref{eq:norm-approximation} as in
\SMcite{O:I:chooseA} gives $x\in E$ and
$A=\Nm_{E/F}(x)$ such that $\gamma=\alpha A$ gives the same
formula for $\lambda$ when $h<t$. When $h=t$, it gives only
$\gamma/\alpha=A+\epsilon$ with $\epsilon\in\OO_F$.
Set $\epsilon=0$ for $h<t$. For $h>0$, $v_E(x)=v_F(A)=-h$;
for $h=0$, $x$ is integral. The norm representative is constructed in
\lean{Odd/UR/NormChoice.lean}{259}{normChoice}.

Using \eqref{eq:odd-log-norm} for the norm pullback of $\lambda$
and \eqref{eq:odd-norm-conversion} on $xP(z)$ gives centered
coefficients
\begin{equation*}
 \begin{aligned}
  B_U&=A+d^p+\epsilon,&
  \theta_U(1-z)&=\Psi_U(B_UP(z)),&&z\in\pp_U^{d_U},\\
  B_E&=A-x+c,&
  \theta_E(1-z)&=\Psi_E(B_EP(z)),&&z\in\pp_E^{d_E}.
 \end{aligned}
\end{equation*}
At $h=t$, the
symmetric-function estimates give
$1-\Nm_{E/F}(1-z)\in\pp_F^T$ for $z\in\pp_E^{d_E}$.
Thus the error $\epsilon\in\OO_F$ does not change the character value,
and $B_E=A-x+c$ remains a valid coefficient.
This is \SMcite{O:I:actualcoeff}, or
\lean{Odd/UR/Coefficients.lean}{251}{coefficients}.

For a finite Galois extension $L/M$ and $x\in L$, write
$e_i^{L/M}(x)$ for the $i$th elementary symmetric function of the
$L/M$-conjugates of $x$.
Now take $C=x+d\in K^\times$ and its two norms
\begin{equation*}
 \begin{aligned}
  Z_U&=\Nm_{K/U}(C)
      =d^p+\sum_{i=1}^{p-1}e_i^{E/F}(x)d^{p-i}+A,\\
  Z_E&=\Nm_{K/E}(C)=x^p-x+c.
 \end{aligned}
\end{equation*}
Compatibility gives $\theta_U(Z_U)=\theta_E(Z_E)$.
The differences $Z_U-B_U$ and $Z_E-B_E$ are
\[
 \eta_U=\sum_{i=1}^{p-1}e_i^{E/F}(x)d^{p-i}-\epsilon
 \quad\text{and}\quad
 \eta_E=x^p-A.
\]
If $h<t$, the symmetric-coefficient bounds give
$v_U(\eta_U)\ge T-d_U$ and $v_E(\eta_E)\ge T-d_E$, so both
quadratic corrections in \eqref{eq:odd-centered-change} are one.
For $h=t$, one has
\[
 v_L(\eta_L)\ge T-\ceil{m_L(\theta_L)/2}
 \qquad(L=U,E),
\]
so \eqref{eq:odd-centered-change} applies to both characters. Moreover,
\[
 \eta_U\equiv e_{p-1}^{E/F}(x)d-\epsilon\pmod{\pp_U}.
\]
These coefficient estimates are \SMcite{O:I:normdepths},
formalized by \lean{Odd/UR/NormDepths.lean}{808}{normDepths}.

For the norms $Z_U,Z_E$ above, put
$S=\Tr_{U/F}(Z_U)-\Tr_{E/F}(Z_E)$. In
\eqref{eq:odd-centered-change}, compatibility cancels the values at
$Z_U,Z_E$, and the determinant identity
\eqref{eq:determinant-restriction} equates
$\theta_U(-\alpha)$ and $\theta_E(-\alpha)$. Thus
\begin{equation*}
 \frac{\Delta_U(\theta_U)}{\Delta_E(\theta_E)}
 \in\mu_4\,\Psi_F(-S)
       \frac{\Psi_U(\eta_U^2/(2B_U))}
            {\Psi_E(\eta_E^2/(2B_E))},
 \qquad
 S=pA-\Tr_{E/F}(x^p)+p\Tr_{E/F}(x).
\end{equation*}
The formula for $S$ follows from the power traces of $d$ in
$X^p-X-c$. It is \SMcite{O:I:scalarphase}, or
\lean{Odd/UR/ScalarPhase.lean}{475}{scalarPhase}.
For $h<t$, or for $h=t$ and $p>3$, the bound $S\in\pp_F^T$ is
\SMcite{O:I:Sdepth}; Lean proves it using Newton identities in
\lean{Odd/UR/TraceDepth.lean}{361}{traceDepth}.

At $h=t$ and $p>3$, the correction over $E$ is one. The correction
over $U$ can be computed using
$(e_{p-1}^{E/F}(x)d-\epsilon)^2/(2A)$. The traces of $d$ and
$d^2$ vanish, and the remaining term $p\epsilon^2/(2A)$ belongs
to $\pp_F^T$. This step is
\lean{Odd/UR/Comparison.lean}{40}{unramifiedBoundaryCorrection}.

When $h=t$ and $p=3$, the trace
$\Tr_{U/F}(d^2)=2$ produces a contribution that cannot be
discarded separately from the quadratic corrections.
The cubic calculation \SMcite{O:I:cubicidentity} combines
them using norm conversion \eqref{eq:odd-norm-conversion}:
\[
 \Psi_F(-S)
 \frac{\Psi_U(\eta_U^2/(2B_U))}
      {\Psi_E(\eta_E^2/(2B_E))}=1.
\]
 The cubic reduction is formalized by
\lean{Odd/UR/CubicIdentity.lean}{537}{cubicIdentity}.
The phase cancellation displayed above is proved by
\lean{Odd/UR/CubicPhase.lean}{283}{phase_eq_one},
and the resulting fourth-root assertion by
\lean{Odd/UR/CubicPhase.lean}{300}{ratio_pow_four_eq_one}.

As a result, for $0\le h\le t$, the preceding calculation gives a quotient
in $\mu_4$. The odd-degree identity in \eqref{eq:power-relations}
gives its $p$th power equal to one. Since
$\mu_p\cap\mu_4=\{1\}$, the quotient is one in this range as
well. 

For $h>t$, one has $m_F(\lambda)=T+h\ge2T$.
Choose $\gamma\in F^\times$ such that
\begin{equation}\label{eq:stable-base-coefficient}
 \lambda(1-z)=\psi_F(\gamma z)
 \qquad\bigl(z\in\pp_F^{\ceil{m_F(\lambda)/2}}\bigr).
\end{equation}
The coefficient $\gamma$ in \eqref{eq:stable-base-coefficient} is a
stationary coefficient for both norm pullbacks. Stable twisting and
\eqref{eq:first-main} give
\[
 \begin{aligned}
  \Delta_U(\theta_U)
    &=\chi_U^0((-\gamma)^{-1})\Delta_F(\lambda)^p,\\
  \Delta_E(\theta_E)
    &=\chi_E((-\gamma)^{-1})\Delta_F(\lambda)^p.
 \end{aligned}
\]
The character values agree by \eqref{eq:determinant-restriction}
for $(\chi_U^0,\chi_E)$, so the two local constants are equal.
This proves \SMcite{O:I:stable-comparison}.
The stationary-coefficient assertion and the comparison are formalized by
\lean{Odd/UR/Stable.lean}{193}{stable_common_covector} and
\lean{Odd/UR/Stable.lean}{298}{oddUR_stable_comparison}.


\section{Totally ramified extensions of odd prime-square degree}
\label{sec:odd-total}

Suppose that $K/F$ is totally ramified with Galois group $C_p^2$,
where $p>2$ is the residue characteristic. Choose $L_1/F$ with
smallest lower ramification break among the degree-$p$ intermediate
fields, and let $L_2/F$ be any other such field. Write their breaks
as $t_1\le t_2$, and put
\[t'=t_1+p(t_2-t_1),\qquad T'=t'+1.\]
The breaks of $K/L_1$ and $K/L_2$ are $t'$ and $t_1$,
respectively. Every valuation is normalized on its own field;
in particular $v_K|_{L_i}=p v_{L_i}$ and
$v_{L_i}|_F=pv_F$.

We construct compatible characters of minimal conductor; every
primitive compatible pair is a common twist of such characters. At higher
conductor, Lamprecht's formula reduces the comparison to normalized
residue-field sums and four character-value factors. Trace--norm
estimates cancel these factors, and the First Main Lemma removes the
remaining fourth-root ambiguity.

\subsection{Artin--Schreier coordinates and minimal characters}
\label{subsec:odd-total-models}

The coordinate assertion \SMcite{O:A:prop:AS} gives
\[
 \Delta^p-\Delta=a\in L_2,
 \qquad v_K(\Delta)=v_{L_2}(a)=-t_1,
 \qquad K=L_2(\Delta),\qquad p\nmid t_1.
\]
It is formalized by
\lean{Odd/Total/ASApproximation.lean}{630}{oddAS_approximate_generator}
and \lean{Odd/Total/ASCoordinate.lean}{308}{aSCoordinate}.

Put $e_i=e_i^{K/L_1}(\Delta)$, so
$e_p=\Nm_{K/L_1}(\Delta)$ and $a=\Nm_{K/L_2}(\Delta)$.
The symmetric-coefficient estimates of \SMcite{O:A:lem:coeff} and
\SMcite{O:A:lem:twisted} include
\[
 v_{L_1}(e_i)\ge(p-1)t_2-it_1\qquad(1\le i<p).
\]
Comparing the coefficient of $Z^p$ in the two expressions 
$\Nm_{L_i/F}(\Nm_{K/L_i}(1-Z\Delta))$ gives
\begin{equation}\label{eq:odd-total-trace-norm}
 \Tr_{L_1/F}(e_p)-\Tr_{L_2/F}(a)\in\pp_F^{t_2+1}.
\end{equation}
This is \SMcite{O:A:thm:H14}, formalized by
\lean{Odd/Total/TraceNorm.lean}{941}{traceNorm}.
When $p=3$ and $t_1=t_2$, the proof uses injectivity of the induced trace map
\[
 \Tr_{L_1/F}:
 \pp_{L_1}^{t_1}/\pp_{L_1}^{t_1+1}
 \longrightarrow \pp_F^{t_1}/\pp_F^{t_1+1}.
\]
The companion paper deduces this from the trace-ideal formula
\eqref{eq:trace-conductor} and the one-dimensionality of the two residue
quotients; Lean proves the same kernel statement directly from trace-pairing
duality, together with the trace-ideal formula.

Recall $P(z)=\sum_{j=1}^{p-1}z^j/j$, and put
\[
 m_1=1+t_1+t',\qquad m_2=1+t_1+t_2,\qquad
 q_i=\ceil{m_i/p}.
\]
Choose a nontrivial $\tau_2\in S(L_2/F)$. By
\SMcite{O:P:coefficients}, there is $\alpha_2\in F^\times$ such that
\begin{equation}\label{eq:odd-total-norm-character}
 \tau_2(1-z)=\psi_F(\alpha_2P(z))
 \qquad\bigl(z\in\pp_F^{\ceil{(t_2+1)/p}}\bigr).
\end{equation}
Put $\Psi_L(z)=\psi_L(\alpha_2z)$, and prescribe
characters $R_i$ on $U_{L_i}^{q_i}$ by
\begin{equation}\label{eq:odd-total-model-charts}
 \begin{aligned}
 R_1(1-z)&=\Psi_{L_1}(e_pP(z))&& (z\in\pp_{L_1}^{q_1}),\\
 R_2(1-z)&=\Psi_{L_2}(aP(z))&& (z\in\pp_{L_2}^{q_2}).
 \end{aligned}
\end{equation}
The conductors are $m_1,m_2$.

Sections~7--8 of the companion paper contain the technical arguments
leading to \SMcite{O:M:compat}: the characters in
\eqref{eq:odd-total-model-charts} satisfy
\[
 R_1(\Nm_{K/L_1}(u))=R_2(\Nm_{K/L_2}(u))
 \qquad(u\in U_K^{q_1}).
\]
This compatibility is formalized by
\lean{Odd/Models/Compatibility.lean}{1643}{compatibility}.

The compatibility just proved extends from the prescribed unit groups to
full characters: by \SMcite{O:M:global}, every extension $\chi_2$ of
$R_2$ to $L_2^\times$ admits a compatible extension $\chi_1$ of $R_1$
to $L_1^\times$. Their conductors are $m_i$, and their common norm
pullback is primitive. This extension result is formalized by
\lean{Odd/Models/GlobalExtension.lean}{463}{globalExtension}.

The preceding construction also describes a given pair, not just the
characters $\chi_i$ constructed above. More precisely, let
$(\varphi_1,\varphi_2)$ be a primitive compatible pair with
$\varphi_2|_{U_{L_2}^{m_2}}=1$. \SMcite{O:M:realize} gives
$m_{L_i}(\varphi_i)=m_i$ for $i=1,2$ and allows the element of $K$
and the additive scale to be chosen for this pair: there exist
$\alpha_2'\in\{\alpha_2,2\alpha_2,\ldots,(p-1)\alpha_2\}$,
$\Delta'\in K$, and $\rho\in\Gal(L_1/F)$ such that
\[
 (\Delta')^p-\Delta'\in L_2,\qquad v_K(\Delta')=-t_1,
\]
and
\begin{equation}\label{eq:odd-total-prescribed-pair}
 \begin{aligned}
 \varphi_1^\rho(1-z)
   &=\psi_{L_1}(\alpha_2'\Nm_{K/L_1}(\Delta')P(z))
       &&(z\in\pp_{L_1}^{q_1}),\\
 \varphi_2(1-z)
   &=\psi_{L_2}(\alpha_2'\Nm_{K/L_2}(\Delta')P(z))
       &&(z\in\pp_{L_2}^{q_2}).
 \end{aligned}
\end{equation}
Thus the values of both characters on these unit groups are expressed
using norms of the same element $\Delta'$.  The element $\Delta'$ and scale $\alpha_2'$
need not equal the initial $\Delta$ and $\alpha_2$.
This statement is formalized by
\lean{Odd/Models/Realization.lean}{1087}{realization}.

For these pairs, \SMcite{O:G:minimal} proves
\[
 \Delta_{L_1}(\varphi_1,\psi_{L_1})
 =\Delta_{L_2}(\varphi_2,\psi_{L_2}),
\]
formalized by \lean{Odd/Conductors/Minimal.lean}{318}{minimal}.

\subsection{Norms of the symmetric coefficient}

The higher-conductor comparison requires the norm of
$-e_{p-1}=-e_{p-1}^{K/L_1}(\Delta)$ in terms of the two norm-character
coefficients, not just its valuation. Retain $\Delta,\tau_2,\alpha_2$
and the full character $\chi_1$ constructed above. Choose a nontrivial
$\tau_1\in S(L_1/F)$ and $\alpha_1\in F^\times$ satisfying
\eqref{eq:odd-total-norm-character} with subscript $1$ in place of $2$.
Put $\varrho=\alpha_1/\alpha_2$.
Let $\sigma\in\Gal(L_1/F)$ be the restriction of the automorphism
of $K/L_2$ carrying $\Delta$ to the root congruent to $\Delta+1$
modulo $\pp_K$.

Both $\chi_1^\sigma/\chi_1$ and $\tau_2\circ\Nm_{L_1/F}$ generate
$S(K/L_1)$, but this alone identifies them only up to a nonzero power
modulo $p$. The technical calculations in
\cite[Lemmas~\SMLnum{O:R:comm-log}--\SMLnum{O:R:T}]{SML}
lead to \SMcite{O:R:normalize}, which identifies the characters exactly:
\[
 \frac{\chi_1^\sigma}{\chi_1}=\tau_2\circ\Nm_{L_1/F}.
\]
This identity is formalized by
\lean{Odd/R2/Normalization.lean}{361}{normalization}.

Put
$q=(t_2-t_1)+\ceil{(t_1+1)/p}$ and choose $c\in L_1^\times$ with
$\Nm_{L_1/F}(c)/\varrho\in U_F^{t_1}$ by
\eqref{eq:norm-approximation}. Comparing the values of these two
characters on $U_{L_1}^q$
gives
\[
 -e_{p-1}\equiv\varrho/c\pmod{\pp_{L_1}^{T'-q}}.
\]
Norm filtration and the specified norm of $c$ give
$\Nm_{L_1/F}(-e_{p-1})\equiv\varrho^{p-1}$ modulo $\pp_F^{T'-q}$.
For every integer $m$ satisfying $pm>t_1$, one has $T'-q\ge1+t_2-(p-2)m$, hence
\begin{equation}\label{eq:odd-total-symmetric-norm}
 \Nm_{L_1/F}(-e_{p-1})\equiv\varrho^{p-1}
       \pmod{\pp_F^{1+t_2-(p-2)m}}.
\end{equation}
This is \SMcite{O:R:R2}, formalized by
\lean{Odd/R2/Congruence.lean}{158}{congruence}.

\subsection{Norm representatives and stationary coefficients}

For a primitive compatible pair $(\theta_1,\theta_2)$,
\SMcite{O:G:reduce} gives
\begin{equation*}
 \theta_i=\varphi_i(\lambda\circ\Nm_{L_i/F}),\qquad
 m_{L_i}(\varphi_i)=m_i\quad(i=1,2),
\end{equation*}
where $(\varphi_1,\varphi_2)$ is a minimal compatible pair.
This reduction is formalized by
\lean{Odd/Conductors/Reduction.lean}{231}{reduction}.

Assume now that $m_{L_2}(\theta_2)>m_2$.  Write
$m_F(\lambda)=1+r$, $m=r-t_2$, so $pm>t_1$, and choose
$\gamma$ with
$\lambda(1-z)=\psi_F(\gamma P(z))$ for
$z\in\pp_F^{\ceil{m_F(\lambda)/p}}$, as in \SMcite{O:P:coefficients}.
For this minimal pair, apply \eqref{eq:odd-total-prescribed-pair}.
Replace $\varphi_1$ and $\theta_1$ by the same permitted conjugate,
which leaves both local constants unchanged. Rename the chosen element
$\Delta$, retain $\alpha_2'$ for its additive scale, and put
$e_j=e_j^{K/L_1}(\Delta)$ for $1\le j\le p$, and
$a=\Nm_{K/L_2}(\Delta)$.
Let $\tau_2$ now denote the nontrivial norm character of $L_2/F$
with coefficient $\alpha_2'$, and choose a nontrivial norm character
$\tau_1$ of $L_1/F$ with coefficient $\alpha_1^0$.

For the chosen $\gamma$, \SMcite{O:P:choose} gives
$w_i\in L_i^\times$ and adjusted coefficients $\alpha_i\in F^\times$
satisfying
\[
 \begin{aligned}
 \Nm_{L_1/F}(w_1)&=\gamma/\alpha_1,&
 \alpha_1/\alpha_1^0&\in U_F^{t_1},\\
 \Nm_{L_2/F}(w_2)&=\gamma/\alpha_2,&
 \alpha_2/\alpha_2'&\in U_F^{t_2}.
 \end{aligned}
\]
The evaluation of $\tau_2$ in \eqref{eq:odd-total-norm-character}
and its analogue for $\tau_1$ remain valid with these adjusted
coefficients, on the same respective ideals.
This choice of $\alpha_i,w_i$ is formalized by
\lean{Odd/Parameters/NormChoice.lean}{92}{normChoice}.
Replacing $\alpha_2'$ by $\alpha_2$ also preserves both equalities
in \eqref{eq:odd-total-prescribed-pair} on their stated ideals, after
the conjugation and renaming just specified. This preservation is
formalized by \lean{Odd/Parameters/Preservation.lean}{140}{preservation}.
With $\varrho=\alpha_1/\alpha_2$ for these adjusted coefficients,
\eqref{eq:odd-total-symmetric-norm} applies to the same data.

The conductors of the norm pullbacks exceed $m_i$ because $pm>t_1$.
Consequently,
\[
 m_{L_1}(\theta_1)=1+t'+pm,\qquad
 m_{L_2}(\theta_2)=1+t_2+pm.
\]
With these $\alpha_i,w_i$, \SMcite{O:P:pullback} gives
\[
 (\lambda\circ\Nm_{L_i/F})(1-z)
   =\psi_{L_i}((\gamma-\alpha_iw_i)P(z)),\qquad
 z\in\pp_{L_i}^{\ceil{m_{L_i}(\theta_i)/p}}.\]
This evaluation of the pulled-back character is formalized by
\lean{Odd/Parameters/Pullback.lean}{159}{pullback}.
Multiplying it by the values of $\varphi_i$ in
\eqref{eq:odd-total-prescribed-pair}, with the choices above, gives
\[
 C_1=\gamma-\alpha_1w_1+\alpha_2e_p,\qquad
 C_2=\gamma-\alpha_2w_2+\alpha_2a,
\]
and
\[
 \begin{aligned}
 \theta_i(1-z)&=\psi_{L_i}(C_iP(z))
       &&\bigl(z\in\pp_{L_i}^{\ceil{m_{L_i}(\theta_i)/p}}\bigr),\\
 v_{L_i}(C_i)&=-n_{L_i}(\psi_{L_i})-m_{L_i}(\theta_i).
 \end{aligned}
\]
These are the coefficients and valuations of
\SMcite{O:P:twisted-coefficients}.

\subsection{Stationary factors and their cancellation}

Put $A_1=\gamma+\alpha_2e_p$ and define
\[
 \begin{aligned}
 Q_1&=\frac{\varphi_2(C_2)}{\varphi_1(A_1)},&
 Q_2&=\lambda\!\left(
             \frac{\Nm_{L_2/F}(C_2)}{\Nm_{L_1/F}(A_1)}\right),\\
 Q_3&=\psi_{L_1}(\alpha_1w_1)\theta_1(A_1/C_1),&
 Q_4&=\psi_{L_2}(-\alpha_2w_2).
 \end{aligned}
\]
Equation~\eqref{eq:odd-centered-change}, applied with $\alpha=1$
and $Z=B=C_i$, and the trace--norm congruence
\eqref{eq:odd-total-trace-norm} give
\begin{equation}\label{eq:odd-total-factor-quotient}
 \frac{\Delta_{L_1}(\theta_1)}{\Delta_{L_2}(\theta_2)}
 =\frac{g_{L_1}(\theta_1,C_1)}{g_{L_2}(\theta_2,C_2)}
       Q_1Q_2Q_3Q_4.
\end{equation}
Here $Q_1$ is the minimal-character transfer, $Q_2$ the norm-twist
correction, $Q_3$ the quadratic stationary-phase correction, and $Q_4$
the linear additive correction. The factorization is \SMcite{O:P:four},
formalized by
\lean{Odd/Parameters/FourFactors.lean}{365}{fourFactors}.

The calculations in Section~9 of the companion paper give
explicit evaluations of $Q_1Q_2Q_4$ and $Q_3$.
With $\Psi_F(x)=\psi_F(\alpha_2x)$, put
\[
 \begin{aligned}
 W&=\gamma/\alpha_2,&
 \mathcal N_a&=\Nm_{L_2/F}(a),\\
 d_W&=1+W^{-p}\mathcal N_a,&
 \mathcal T&=\Tr_{K/L_1}((\Delta/w_2)^{p-1}).
 \end{aligned}
\]
Here $v_F(W)=-m$ and $d_W\in U_F^1$.
\SMcite{O:T:Q124} gives
\[
 Q_1Q_2Q_4=\Psi_F\Bigl(
 \frac{W^{2-p}\Nm_{L_1/F}(-e_{p-1})}{d_W}
 -\frac{W^{2-p}\Nm_{L_1/F}(-e_{p-1})}{2d_W^2}
 +\frac{W^{1-p}\mathcal N_a\Nm_{L_1/F}(\mathcal T)}{2d_W^2}\Bigr),
\]
formalized by \lean{Odd/Transition/Q124.lean}{385}{q124}.
\SMcite{O:T:Q3} gives
\[
 Q_3=\Psi_F\!\left(-\frac W2\frac{(\varrho/W)^{p-1}}{d_W}\right),
\]
formalized by \lean{Odd/Transition/Q3.lean}{294}{q3}.

By \SMcite{O:E:main} and the congruence
\eqref{eq:odd-total-symmetric-norm}, the sum of the two displayed
additive-character arguments belongs to $\pp_F^{1+t_2}$, where
$\Psi_F$ is trivial. Hence $Q_1Q_2Q_3Q_4=1$.
The Lean assembly is
\lean{Odd/Transition/Phase.lean}{82}{phase} and
\lean{Odd/Transition/Cancel.lean}{49}{cancel}.
Thus \eqref{eq:odd-total-factor-quotient} is a quotient of centered
normalized quadratic sums, hence lies in $\mu_4$; its $p$th power is one by
\eqref{eq:power-relations}, so it is one.

The minimal- and higher-conductor ranges exhaust the primitive
compatible pairs.  Comparing each intermediate field with the chosen
smallest-break field gives \SMcite{O:G:totalbranch}, formalized by
\lean{Odd/Total/Comparison.lean}{467}{comparison}.

\section{Dyadic extensions with an unramified intermediate field}\label{sec:dyadic-ur}
Let $F$ have residue characteristic two, let $U/F$ be unramified
quadratic, and let $E/F$ be ramified quadratic of break $t\ge1$.
Put $K=UE$ and $T=t+1$, and write $k\subset\kappa$ for the residue
fields of $F$ and $U$. Let $\omega_U,\omega_E$ be the nontrivial
norm characters. For a primitive compatible pair $\theta_U,\theta_E$,
the desired comparison is
\[
 \Delta_U(\theta_U)\Delta_F(\omega_U)
 =\Delta_E(\theta_E)\Delta_F(\omega_E).
\]
The square relation \eqref{eq:quadratic-square} determines the
induction expressions only up to sign. We resolve this ambiguity by
comparing stationary coefficients and critical functions in the
minimal, intermediate, and stable conductor ranges.

\subsection{The norm character and compatible models}

Stationary linearization for $\omega_E$ gives $\alpha\in F^\times$
such that, with $\Psi_L(x)=\psi_L(\alpha x)$,
\begin{equation*}
 \omega_E(1-z)=\Psi_F(z)\quad
       (z\in\pp_F^{\ceil{T/2}}),\qquad
 v_F(\alpha)=-n_F(\psi_F)-T.
\end{equation*}
The largest trivial ideal of each $\Psi_L$, for $L=F,U,E,K$, is
$\pp_L^T$. Quadratic norm expansion and triviality of $\omega_E$
on norms give
\begin{equation}\label{eq:dur-norm-trace}
 \Psi_E(z)=\Psi_F(\Nm_{E/F}(z))
                    \qquad(z\in\pp_E^{\ceil{T/2}}).
\end{equation}
We call this the \emph{dyadic norm-conversion formula}.
For odd $T$, it does not apply on
$\pp_E^{\floor{T/2}}$; the reciprocity identity
\eqref{eq:dur-critical-reciprocity} will handle that depth.
The mathematical statement is \SMcite{U:quadratic-conversion};
the formal declaration is
\lean{Dyadic/UR/NormPhase.lean}{72}{normPhase}.

By \SMcite{D:UR:oddT}, an odd quadratic conductor occurs only in mixed
characteristic and has the form $T=2e+1$, where $e=v_F(2)$. In that
case
\begin{equation*}
 \Psi_F(4z)=(-1)^{\Tr_{k/\F_2}(\bar z)}
                            \qquad(z\in\OO_F).
\end{equation*}
This formula is formalized by
\lean{Dyadic/UR/OddConductor.lean}{282}{oddConductor}.

Choose $c\in\OO_F$ with $\Tr_{k/\F_2}(\bar c)=1$, and choose
$d\in\OO_U$ with $d^2-d=c$. For the nontrivial automorphism
$\sigma$ of $U/F$, one has $\sigma(d)=1-d$ and
$\Tr_{U/F}(d)=1$. The construction in \SMcite{D:UR:models} gives a
\emph{model pair} $(\chi_U^0,\chi_E^0)$: characters of
$U^\times$ and $E^\times$ of conductor $T$, with a primitive
common norm pullback, satisfying
\begin{equation}\label{eq:dur-models}
 \begin{aligned}
  (\chi_U^0)^\sigma/\chi_U^0&=\omega_E\circ\Nm_{U/F},\\
  \chi_U^0(1-z)&=\Psi_U(d^2z)
                         &&(z\in\pp_U^{\ceil{T/2}}),\\
  \chi_E^0(1-z)&=\Psi_E(cz)
                         &&(z\in\pp_E^{\ceil{T/2}}).
 \end{aligned}
\end{equation}
The pair in \eqref{eq:dur-models} is constructed as in
Section~\ref{subsec:odd-ur-models}. Here dyadic norm conversion
\eqref{eq:dur-norm-trace} gives the coefficient $d^2-d=c$ for
$\chi_E^0$. The construction is formalized by
\lean{Dyadic/UR/Models.lean}{887}{models}.

The common-twist construction of Section~\ref{subsec:odd-ur-models}
gives a quasi-character $\lambda$ of $F^\times$.
Put $\chi_E=\theta_E/(\lambda\circ\Nm_{E/F})$; it satisfies
the $E$-formula in \eqref{eq:dur-models}.
Thus \SMcite{D:UR:actual} gives the analogue of
\eqref{eq:odd-ur-conductors} with $p=2$:
\begin{equation}\label{eq:dur-twist}
 \begin{gathered}
  \begin{aligned}
   \theta_U&=\chi_U^0(\lambda\circ\Nm_{U/F}),&
   m_U(\theta_U)&=T+h,\\
   \theta_E&=\chi_E(\lambda\circ\Nm_{E/F}),&
   m_E(\theta_E)&=T+2h,
  \end{aligned}\\
  h\ge0,\qquad
  \begin{cases}
   m_F(\lambda)=T+h,&h>0,\\
   m_F(\lambda)\le T,&h=0.
  \end{cases}
 \end{gathered}
\end{equation}
The formal result is
\lean{Dyadic/UR/Twist.lean}{386}{twist}.

\subsection{Norm representatives and the minimal comparison}

For $0\le h<T$, put
$s_U=\ceil{(T+h)/2}, s_E=\ceil{T/2}+h$.
Stationary linearization and norm approximation give
$x\in E$ and $A=\Nm_{E/F}(x)$ such that
$\lambda(1-z)=\Psi_F(Az)$ for $(z\in\pp_F^{s_U})$.
Here $v_E(x)=-h$ for $h>0$, while $x$ may be taken integral for
$h=0$.  This is \SMcite{D:UR:choosex}, formalized by
\lean{Dyadic/UR/NormChoice.lean}{247}{normChoice}.

The model formulas and quadratic norm conversion give stationary
coefficients
\[B_U=A+d^2,\qquad B_E=A-x+c\]
for $\theta_U,\theta_E$ on $\pp_U^{s_U},\pp_E^{s_E}$.
Their conductor and valuation assertions are
\SMcite{D:UR:coeffs}, formalized by
\lean{Dyadic/UR/Coefficients.lean}{336}{coefficients_exists}.

As in Section~\ref{sec:odd-centered}, use norms of the same element
so that compatibility equates the character values. Put
\[
 C=x+d,\qquad Z_U=\Nm_{K/U}(C),\qquad
 Z_E=\Nm_{K/E}(C),\qquad b=\Tr_{E/F}(x).
\]
Then $\theta_U(Z_U)=\theta_E(Z_E)$.  

At $h=0$, one has $m_U(\theta_U)=m_E(\theta_E)=T$.
Both $Z_U$ and $Z_E$ are valid stationary coefficients for
$\theta_U$ and $\theta_E$, respectively, and both replacements are
valid in \eqref{eq:normalized-factor}.
The quotient of the additive factors in that formula is $\Psi_F(-S)$,
where
\[
 S=\Tr_{U/F}(Z_U)-\Tr_{E/F}(Z_E)
   =1+4c-(b^2-4A).
\]
The comparison proved in \SMcite{D:UR:minimal} is
\[
 \Delta_U(\theta_U)\Delta_F(\omega_U)
 =\Delta_E(\theta_E)\Delta_F(\omega_E).
\]
For even $T$, there are no residual sums, and the remaining factor is
$(-1)^T\Psi_F(1-S)=1$. For odd $T$, use
$W_E=\Nm_{E/F}(2x+1)$ as the stationary coefficient for $\omega_E$.
Including the norm-character factors changes the scalar prefactor to
\[
 (-1)^T\Psi_F(W_E-S)
 =(-1)^T\Psi_F\bigl(\Nm_{U/F}(b+2d)\bigr)=1.
\]
With normalized critical sums $g_U,g_E,g_{\omega_E}$ for
$\theta_U,\theta_E,\omega_E$ at $Z_U,Z_E,W_E$ relative to
$\Psi_U,\Psi_E,\Psi_F$, respectively, use \eqref{eq:normalized-factor}.
The changes $Z_L-B_L$ are incorporated in the critical sums $g_U,g_E$
through their stationary coefficients $Z_L$. The proof obtains
\[
 \frac{\Delta_U(\theta_U)\Delta_F(\omega_U)}
      {\Delta_E(\theta_E)\Delta_F(\omega_E)}
 =\frac{g_U}{g_Eg_{\omega_E}},
 \qquad g_U=g_Eg_{\omega_E}.
\]
The complete local-constant equality, for both parities of $T$,
is formalized by \lean{Dyadic/UR/Minimal.lean}{756}{minimal}.

\subsection{Higher conductors and comparison of critical sums}

For $1\le h<T$, let $x'$ be the nontrivial conjugate of $x$ over
$F$.  The corrected stationary coefficient of
\SMcite{D:UR:correctedreps} is
\[
 \widehat Z_E=(x'/x)Z_E,\qquad
 \widehat Z_E-B_E=b(1-c/x),
\]
while $Z_U$ is the corresponding coefficient over $U$.
Let $g_U,g_E,g_{\omega_E}$ be the normalized sums in
\eqref{eq:critical-function} for $\theta_U,\theta_E,\omega_E$ at
$Z_U,\widehat Z_E,1$, respectively, with value one in even conductor.
The character-value and trace calculation gives
\[
 \frac{\Delta_U(\theta_U)\Delta_F(\omega_U)}
      {\Delta_E(\theta_E)\Delta_F(\omega_E)}
 =\frac{g_U}{g_Eg_{\omega_E}}.
\]
Indeed, the parity calculation of \SMcite{D:UR:traceparity} makes
the scalar prefactor equal to one; this reduction is
\SMcite{D:UR:criticalratio}, formalized by
\lean{Dyadic/UR/CriticalRatio.lean}{355}{criticalRatio}.

The remaining normalized sums are evaluated directly in the companion
paper. If $T+h$ is odd, \SMcite{D:UR:resU} gives $g_U=1$;
if $T+h$ is even, this follows from the even-conductor formula.
If $T$ is even, then $g_E=g_{\omega_E}=1$ as well.  If $T=2e+1$ is odd, the
critical reciprocity identity
\begin{equation}\label{eq:dur-critical-reciprocity}
 \omega_E(1+\Nm_{E/F}(z))=\Psi_E(z)
                              \qquad(z\in\pp_E^e)
\end{equation}
identifies the two ramified critical functions after the Frobenius
residue substitution, and \SMcite{D:UR:resE} gives
$g_E=g_{\omega_E}^{-1}$.  These two propositions are formalized by
\lean{Dyadic/UR/UnramifiedSum.lean}{439}{unramifiedSum} and
\lean{Dyadic/UR/RamifiedSum.lean}{512}{ramifiedSum}.
Thus the quotient is one for every $1\le h<T$.

For $h\ge T$, one has $m_F(\lambda)=T+h\ge2T$.
Choose $\gamma$ for $\lambda$ as in \eqref{eq:stable-base-coefficient}.
Using \SMcite{D:UR:highcov} to evaluate the model characters at
$(-\gamma)^{-1}$, \SMcite{D:UR:stable} gives
\[
 \begin{aligned}
  \mathcal A_U(\theta_U)
    &=\chi_U^0((-\gamma)^{-1})\omega_U((-\gamma)^{-1})
      \Delta_F(\lambda)^2,\\
  \mathcal A_E(\theta_E)
    &=\chi_E((-\gamma)^{-1})\omega_E((-\gamma)^{-1})
      \Delta_F(\lambda)^2.
 \end{aligned}
\]
The right sides agree by \eqref{eq:determinant-restriction} for
$(\chi_U^0,\chi_E)$. This comparison is formalized by
\lean{Dyadic/UR/Stable.lean}{37}{stable}.
The alternative in \SMcite{D:UR:FML-free-alternative} chooses
$\gamma$ as in \eqref{eq:stable-base-coefficient} and
$u\in U_F^t$ with
$\gamma/(\alpha u)\in\Nm_{E/F}(E^\times)$.
Replacing $\alpha$ by $\alpha u$ then extends the intermediate
argument to every $h>0$; the Lean proof uses stable twisting for $h\ge T$.

The minimal, intermediate, and stable ranges exhaust $h\ge0$.
Comparing each ramified field with $U$ gives all three comparisons;
the resulting declaration is
\lean{Dyadic/UR/Comparison.lean}{235}{comparison}.

\section{Totally ramified biquadratic extensions in characteristic two}\label{sec:dyadic-equal}
Let $F$ have characteristic two and residue field $k$, and let
$K/F$ be totally ramified biquadratic. Write $L_1,L_2,L_3$ for
the quadratic intermediate fields and $\omega_i$ for their nontrivial
norm characters. Fix a primitive compatible family $(\theta_i)$
with common norm pullback $\Theta$.
Artin--Schreier equations supply an element of $K$ whose norms
prescribe compatible characters of minimal conductor. At those
conductors, the $\theta_i$ have odd conductor and the $\omega_i$
have even conductor, so only the $\theta_i$ contribute
nontrivial residue-field factors. 

\subsection{Artin--Schreier characters and a common origin}

Choose a uniformizer $\pi$ of $K$, and use the compatible
uniformizers
\begin{equation*}
 \Pi_i=\Nm_{K/L_i}(\pi),\qquad
 \varpi=\Nm_{K/F}(\pi)=\Nm_{L_i/F}(\Pi_i).
\end{equation*}
The coefficient field identifies $F$ with $k((\varpi))$. Begin
with the additive character
\[
 \psi_F(x)=(-1)^{\Tr_{k/\F_2}(\Res_F(x\,d\varpi))},
\]
whose largest trivial ideal is $\OO_F$. The existence of the
Laurent-series presentation is formalized by
\lean{Residues/EqualCharacteristicPresentation.lean}{574}{existsEqualCharacteristicPresentation}.
Additive scaling will give the result for an arbitrary nontrivial
character of $F$.

For $z^2+z=f$ defining a quadratic extension, the norm character is
\begin{equation}\label{eq:deq-symbol}
 \omega_f(u)=(-1)^{\Tr_{k/\F_2}(\Res_F(f\,du/u))},
                        \qquad u\in F^\times.
\end{equation}
The \emph{Artin--Schreier invariance identity} underlying \SMcite{D:EQ:ASsymbol},
\begin{equation*}
 \Tr_{k/\F_2}\Res_F\bigl((w^2+w)\,du/u\bigr)=0
                  \qquad(w\in F,\ u\in F^\times),
\end{equation*}
shows that replacing $f$ by $f+w^2+w$ does not change
\eqref{eq:deq-symbol}. Trace compatibility makes $\omega_f$
trivial on norms from the quadratic extension.
For reduced $f$ with largest pole $t>0$, the conductor is $t+1$,
as formalized by
\lean{EqualChar/NormCharacter.lean}{592}{constructedNormCharacter_conductor}.

The Artin--Schreier invariance identity is formalized by
\lean{EqualChar/NormCharacter.lean}{53}{residueExponent_artinSchreier}.
The identification of the residue character \eqref{eq:deq-symbol}
with the nontrivial quadratic norm character is formalized by
\lean{Characters/ArtinSchreierProduct.lean}{199}{artinSchreierNormCharacter_eq_normCharacter}.
The classical residue inputs are
\lean{Residues/TraceCompatibility.lean}{704}{traceCompatibility}
and \lean{Residues/LogNorm.lean}{393}{logNorm}.
Their constructions use
\lean{Residues/Substitution.lean}{679}{substitution},
\lean{Residues/FreeBasis.lean}{1032}{freeBasis}, and
\lean{Residues/TraceSubstitution.lean}{892}{traceSubstitution}.
These residue and trace identities are proved in \cite[Appendix~B]{SML}.

Label $L_1$ by the smallest break and choose reduced equations
\[
 z_i^2+z_i=f_i\quad(i=1,2),\qquad
 z_3=z_1+z_2,\qquad f_3=f_1+f_2.
\]
Write the breaks as $t_1\le t_2=t_3$ and put
$T_i=t_i+1=2r_i$, with $r_3=r_2$. In the equal-break case,
$f_1,f_2$ have distinct leading pole coefficients.

Let $\zeta_i$ be $\varpi^{t_2}$ times the negative-power part of $f_i$.
Choose its square root $h_i\in F$, so that
\begin{equation*}
 \zeta_i=h_i^2,\qquad h_3=h_1+h_2,\qquad Y=h_1z_2+h_2z_1.
\end{equation*}
One has $v_F(h_1)=r_2-r_1$, while $h_2,h_3$ are units.
For the nontrivial $g_i\in\Gal(K/L_i)$,
\begin{equation*}
 g_i(Y)-Y=\Tr_{K/L_i}(Y)=h_i,\qquad
 a_i=\Nm_{K/L_i}(Y),\qquad
 v_K(Y)=v_{L_i}(a_i)=-t_1.
\end{equation*}
We call $Y$ the \emph{common origin}: its norms $a_i$ prescribe
the unit characters below, and the norms $\zeta_i$ of its traces
supply the coefficients for $\omega_i$.
These identities and valuations are \SMcite{D:EQ:origin-identities},
formalized by
\lean{Dyadic/Equal/Origin.lean}{1483}{exact_origin}
and
\lean{Dyadic/Equal/Origin.lean}{1538}{origin_orders}.

The relative breaks of $K/L_i$ are $2t_2-t_1,t_1,t_1$.
Choose $\alpha=\varpi^{-T_2}$ and put
$\Psi_F=\psi_F(\alpha\,\cdot)$ and
$\Psi_{L_i}=\Psi_F\circ\Tr_{L_i/F}$.
The largest trivial ideals of $\Psi_F$ and $\Psi_{L_i}$ are
$\pp_F^{T_2}$ and $\pp_{L_i}^{J_i}$, where
\[
 J_1=2T_2-T_1,\qquad J_2=J_3=T_2.
\]
For $i=1,2,3$,
\begin{equation}\label{eq:deq-lower-charts}
 \begin{aligned}
  \omega_i(1-z)&=\Psi_F(\zeta_i z)
                      &&(z\in\pp_F^{\ceil{T_i/2}}),\\
  \Psi_{L_i}(\zeta_i z)&=\Psi_F(\zeta_i\Nm_{L_i/F}(z))
                      &&(z\in\pp_{L_i}^{\ceil{T_i/2}}).
 \end{aligned}
\end{equation}
The first identity is \SMcite{D:EQ:lower}. The second follows from
\eqref{eq:dur-norm-trace} with additive scale $\alpha\zeta_i$.
Both identities are formalized by
\lean{Dyadic/Equal/Origin.lean}{1039}{dyadicEqual_lowerCharts}.

\subsection{Compatible characters and stationary coefficients}
\label{subsec:deq-stationary}

Put $m_1=2T_2-1$, $m_2=m_3=T_1+T_2-1$, and
$s_i=\ceil{m_i/2}$. Define the prescribed unit characters on
$U_{L_i}^{s_i}$ by
\begin{equation}\label{eq:deq-model-charts}
 R_i(1-z)=\Psi_{L_i}(a_i z)\qquad(z\in\pp_{L_i}^{s_i}).
\end{equation}
The commutator identity \SMcite{D:EQ:commmodel} and the common norm
pullback \SMcite{D:EQ:Rcompat} allow character extension and
Hilbert~90 descent, producing a primitive compatible model family
$(\chi_i^0)$ of conductors $m_i$; this is
\SMcite{D:EQ:core-exists}.  The commutator is formalized in
\lean{Dyadic/Equal/Commutator.lean}{709}{commutator}.
Here the companion paper derives the identity directly from
Artin--Schreier invariance, whereas Lean uses
\lean{Dyadic/Equal/CartierRational.lean}{85}{cartierRational}.

For the given primitive compatible family, \SMcite{D:EQ:realize}
gives a quasi-character $\lambda$ of $F^\times$ and primitive compatible
characters $\chi_i$ of conductors $m_i$ such that, for $i=1,2,3$,
\begin{equation}\label{eq:deq-twist-conductors}
 \begin{gathered}
  \theta_i=\chi_i(\lambda\circ\Nm_{L_i/F}),\qquad
  \chi_i|_{U_{L_i}^{s_i}}=R_i,\\
  m_{L_i}(\theta_i)=
  \begin{cases}
   m_i,&n\le T_2+r_1-1,\\
   2n-T_i,&n\ge T_2+r_1,
  \end{cases}
  \qquad n=m_F(\lambda).
 \end{gathered}
\end{equation}
The model conductors are odd and the larger norm-pullback conductors
are even, so equal-conductor cancellation cannot occur.  The formal
realization is \lean{Dyadic/Equal/Twist.lean}{296}{twist}.

For $n\le T_2+r_1-1$, the proof of \SMcite{D:EQ:min-origin} chooses
$X\in\pp_{L_3}^{1-r_1}$ such that, with $A=\Nm_{L_3/F}(X)$,
\[
 \lambda(1-z)=\Psi_F(Az)\qquad(z\in\pp_F^{r_1+r_2}).
\]
Take $X=0$ when $n\le r_1+r_2$.
Put $C=Y+X$. Then $v_K(C)=-t_1$, and its traces to $L_1,L_2$
are nonzero.

For $T_2+r_1\le n\le2T_2+r_1-2$, the proof of
\SMcite{D:EQ:intermediate} chooses $X\in L_3^\times$ with
$v_{L_3}(X)=T_2-n$ such that
\[
 \lambda(1-z)=\Psi_F(\Nm_{L_3/F}(X)z)
 \qquad\bigl(z\in\pp_F^{\floor{(n+r_1)/2}}\bigr).
\]
Put $C=Y+X$; now $v_K(C)=2(T_2-n)$.

For $n\le2T_2+r_1-2$ and $i=1,2$, the stationary coefficient of
$\theta_i$ is $\Nm_{K/L_i}(C)$ in both ranges. The coefficient used
for $\omega_i$ is the norm of the trace of $C$ in the minimal range,
but that of $Y$ in the intermediate range:
\begin{equation}\label{eq:deq-stationary}
 \begin{gathered}
  \theta_i(1-z)=\Psi_{L_i}(\Nm_{K/L_i}(C)z),
  \qquad z\in\pp_{L_i}^{\ceil{m_{L_i}(\theta_i)/2}},\\[1mm]
  \omega_i(1-z)=
  \begin{cases}
   \Psi_F(\Nm_{L_i/F}(\Tr_{K/L_i}(C))z),&n\le T_2+r_1-1,\\
   \Psi_F(\zeta_i z),&T_2+r_1\le n\le2T_2+r_1-2,
  \end{cases}
  \qquad z\in\pp_F^{\ceil{T_i/2}}.
 \end{gathered}
\end{equation}
The depth in the $\theta_i$-evaluation is $s_i$ in the minimal range
and $n-r_i$ in the intermediate range. The minimal-range assertions
are \SMcite{D:EQ:min-origin}, formalized by
\lean{Dyadic/Equal/MinimalOrigin.lean}{810}{minimalOrigin}.
When $t_1=t_2$, both minimal-range evaluations also hold for $i=3$.

In the intermediate range, the bounds of \SMcite{D:EQ:bounds} permit
\cite[(\SMLnum{D:EQ:pullback-adjust})]{SML} and the subsequent norm
expansion on $\pp_{L_i}^{n-r_i}$. The resulting evaluation of
$\lambda\circ\Nm_{L_i/F}$ is formalized by
\lean{Dyadic/Equal/Intermediate.lean}{351}{intermediate_pullback_chart}.
Multiplication by \eqref{eq:deq-model-charts} gives
\eqref{eq:deq-stationary} for $\theta_i$.

Put $p_X=\Tr_{L_3/F}(X)$. The trace identity in
\SMcite{D:EQ:bounds} gives
\begin{equation}\label{eq:deq-trace-norm-correction}
 \Nm_{L_i/F}(\Tr_{K/L_i}(C))=\zeta_i+p_X^2,\qquad i=1,2.
\end{equation}
The underlying norm and trace identities are formalized by
\lean{Dyadic/Equal/ThirdNorm.lean}{459}{dyadicEqual_thirdNorm_identities}.

For $n\ge2T_2+r_1-1$, choose $\gamma$ for $\lambda$ as in
\eqref{eq:stable-base-coefficient}, and put $A=\gamma/\alpha$.
Norm approximation \eqref{eq:norm-approximation} gives
$x_i\in L_i^\times$ with
$\Nm_{L_i/F}(x_i)/(A/\zeta_i)\in U_F^{t_i}$ for $i=1,2$.
The proof of \SMcite{D:EQ:highcomparison} gives
\begin{equation}\label{eq:deq-stable-coefficients}
 \begin{gathered}
  (\lambda\circ\Nm_{L_i/F})(1-z)
    =\Psi_{L_i}((A+\zeta_ix_i)z),
  \qquad z\in\pp_{L_i}^{n-r_i},\\[1mm]
  \chi_i\!\left(\bigl(\alpha(A+\zeta_ix_i)\bigr)^{-1}\right)
    =\chi_i(\gamma^{-1})
       \Psi_F\!\left(\frac{\Nm_{K/F}(Y)}{A}\right).
 \end{gathered}
\end{equation}
The coefficient has valuation $2(T_2-n)$. Since $n-r_i\ge m_i$,
$\chi_i$ is trivial on $U_{L_i}^{n-r_i}$, so $A+\zeta_ix_i$ is also
a stationary coefficient for $\theta_i$. The coefficient for
$\omega_i$ remains $\zeta_i$, as in \eqref{eq:deq-lower-charts}.
This construction is formalized by
\lean{Dyadic/Equal/HighCoefficient.lean}{490}{highCoefficient}.

\subsection{Comparison at minimal and higher conductor}
\label{subsec:deq-comparison}

For $n\le2T_2+r_1-2$, use $C$ and the stationary coefficients in
\eqref{eq:deq-stationary}, and let $g_{\theta_i},g_{\omega_i}$ be
their normalized critical sums as in \eqref{eq:critical-function}.
Since $T_i$ is even, $g_{\omega_i}=1$. In the minimal range,
\eqref{eq:common-origin} applies; in the intermediate range, use
\eqref{eq:normalized-factor}, \eqref{eq:biquadratic-e2}, and
\eqref{eq:deq-trace-norm-correction}. Thus, for $i=1,2$,
\begin{equation}\label{eq:deq-induction-cases}
 \begin{aligned}
  \mathcal A_{L_i}(\theta_i)
  &=D_\theta(\alpha^{-1})\Theta(C)^{-1}\\
  &\quad\times
  \begin{cases}
   \Psi_F(e_2(C))g_{\theta_i},&n\le T_2+r_1-1,\\
   \Psi_F(e_2(C)+p_X^2),&T_2+r_1\le n\le2T_2+r_1-2.
  \end{cases}
 \end{aligned}
\end{equation}
In the intermediate range the conductors $2n-T_i$ are even, so
$g_{\theta_i}=1$ as well.

For $n\le T_2+r_1-1$, let $H_{\theta_i}$ be the critical function
for these coefficients with critical element $\Pi_i^{\floor{m_i/2}}$.
Let $\rho_i:k\to k$ be induced by $\Nm_{K/L_i}$ on the successive
unit quotients at depths $t_2$ in $K$ and $\floor{m_i/2}$ in $L_i$,
using $\pi,\Pi_i$ to identify the quotients with $k$.
Varying $C$ by $C_z=C(1+\pi^{t_2}z)$, for $z\in\OO_K$, gives
$H_{\theta_i}(\rho_i(\bar z))=Q(\bar z)$,
independent of $i$. If $t_1<t_2$, the two maps $\rho_i$ are
bijective, so reindexing gives $g_{\theta_1}=g_{\theta_2}$.

If $t_1=t_2=t$, put
\[
 \lambda_i=
 \overline{\frac{g_i(\pi)-\pi}{\pi^{t+1}}}\in k^\times.
\]
By \cite[Lemmas~D.2--D.3]{SML}, the three $\lambda_i$ are distinct,
and \cite[Proposition~3.4]{FML} gives
\[
 \rho_i(z)=z^2+\lambda_i z,\qquad
 \ker\rho_i=\{0,\lambda_i\}.
\]
Lean computes the same ramification coefficients in origin
coordinates. The common pullback and \eqref{eq:missing-coset}
give equality of all three normalized sums, formalized by
\lean{Dyadic/Equal/EqualBreaks.lean}{1437}{equalBreaks}.
The first case of \eqref{eq:deq-induction-cases} therefore gives
the minimal comparison \SMcite{D:EQ:min-two}.

For $T_2+r_1\le n\le2T_2+r_1-2$, the second case of
\eqref{eq:deq-induction-cases} is independent of $i$, including the
factor $\Psi_F(p_X^2)$. This proves \SMcite{D:EQ:intermediate},
formalized by \lean{Dyadic/Equal/Intermediate.lean}{744}{intermediate}.

For $n\ge2T_2+r_1-1$, use the coefficient and character-value
identities in \eqref{eq:deq-stable-coefficients}. Stable twisting and
\eqref{eq:first-main}, as in \SMcite{D:EQ:highcomparison}, give
\[
 \mathcal A_{L_i}(\theta_i)
 =\chi_i(\gamma^{-1})\omega_i(\gamma^{-1})
   \Psi_F\!\left(\frac{\Nm_{K/F}(Y)}{A}\right)
   \Delta_F(\lambda)^2,\qquad i=1,2.
\]
By \eqref{eq:determinant-restriction} for $(\chi_i)$, the right side
is independent of $i$. This comparison is formalized by
\lean{Dyadic/Equal/HighComparison.lean}{318}{highComparison}.

The alternative in \SMcite{D:EQ:FML-free-alternative} chooses
$\gamma$ as in \eqref{eq:stable-base-coefficient} and
$u\in U_F^{T_2-1}$ with
$\gamma/(\alpha u)\in\Nm_{L_3/F}(L_3^\times)$.
After replacing $\alpha$ by $\alpha u$, the intermediate argument
applies to all $n\ge T_2+r_1$. The Lean proof uses the three ranges
above.

The three ranges are consecutive and exhaust the compatible family;
relabeling gives the third intermediate field. The resulting theorem
is \SMcite{D:EQ:main}, formalized by
\lean{Dyadic/Equal/Comparison.lean}{512}{comparison}.

\section{Totally ramified mixed characteristic with nonmaximal breaks}\label{sec:dyadic-nonmaximal}
Let $F/\Q_2$ be finite, let $e=v_F(2)$, and let $K/F$ be
totally ramified biquadratic. Write $k$ for the residue field.
Label the quadratic intermediate fields $L_i$ so that their breaks over $F$ are
\[
 t_1=2r_1-1,\qquad t_2=t_3=2r_2-1,\qquad
 1\le r_1\le r_2\le e.
\]
Choose a uniformizer $\pi$ of $K$, and use the compatible
uniformizers
\begin{equation*}
 \Pi_i=\Nm_{K/L_i}(\pi),\qquad
 \varpi=\Nm_{K/F}(\pi)=\Nm_{L_i/F}(\Pi_i).
\end{equation*}
Put $r_3=r_2$ and $T_i=t_i+1$, let $\omega_i$ be the nontrivial norm
character of $L_i/F$, and let $g_i$ be the nontrivial element of
$\Gal(K/L_i)$. For a primitive compatible family $\theta_i$
with common pullback $\Theta$, \SMcite{D:NM:main} asserts equality
of $\Delta_{L_i}(\theta_i)\Delta_F(\omega_i)$.
The conductor pattern is the same as in
Section~\ref{sec:dyadic-equal}: the model characters have odd
conductors, and their higher-conductor twists have even conductors.

\subsection{Aligned quadratic generators and compatible characters}

Choose $z_i\in L_i$ with
$z_i^2+z_i=f_i\in F$ and $v_F(f_i)=1-2r_i$ for $i=1,2$.

To construct the common element $Y$, modify $z_1$ within $L_1$
and choose $d\in F$ as in \SMcite{D:NM:alignment}, so that
\begin{equation*}
 v_F(d)=r_2-r_1,\qquad
 E=f_1+d^2f_2\in\pp_F^{1-r_1},\qquad 1+d\in\OO_F^\times.
\end{equation*}
We call this choice of $z_1,z_2,d$ the \emph{alignment}.
Its formal declaration is \lean{Dyadic/Nonmaximal/Alignment.lean}{868}{alignment}.

Set $Y=z_1+dz_2$ and retain the notation
$a_i=\Nm_{K/L_i}(Y)$, $h_i=\Tr_{K/L_i}(Y)$,
$\zeta_i=\Nm_{L_i/F}(h_i)$. The traces and their norms are
\begin{equation*}
 \begin{aligned}
 h_1&=2z_1-d,&\zeta_1&=d^2+2d-4f_1,\\
 h_2&=2dz_2-1,&\zeta_2&=1+2d-4d^2f_2,\\
 h_3&=-(1+d),&\zeta_3&=(1+d)^2.
 \end{aligned}
\end{equation*}
The resulting valuations are
\[
 v_K(Y)=v_{L_i}(a_i)=-t_1,\qquad
 v_{L_1}(h_1)=2(r_2-r_1),\qquad
 v_{L_2}(h_2)=v_{L_3}(h_3)=0.
\]
Also $\Tr_{K/F}(Y)=-2(1+d)$.
These data are \SMcite{D:NM:origin-identities}, formalized by
\lean{Dyadic/Nonmaximal/Origin.lean}{877}{dyadicNonmaximal_origin_data}.
 
Choose $\alpha\in F^\times$ with
$\omega_2(1-z)=\psi_F(\alpha\zeta_2z)$ on $\pp_F^{r_2}$,
and put $\Psi_F=\psi_F(\alpha\,\cdot)$ and
$\Psi_{L_i}=\Psi_F\circ\Tr_{L_i/F}$.
Then $v_F(\alpha)=-n_F(\psi_F)-T_2$. The largest trivial ideals
of $\Psi_F$ and $\Psi_{L_i}$ are $\pp_F^{T_2}$ and
$\pp_{L_i}^{J_i}$, where
\[
 J_1=2T_2-T_1,\qquad J_2=J_3=T_2.
\]
For $i=1,2,3$,
\begin{equation}\label{eq:dnm-lower-charts}
 \begin{aligned}
  \omega_i(1-z)&=\Psi_F(\zeta_i z)
                      &&(z\in\pp_F^{\ceil{T_i/2}}),\\
  \Psi_{L_i}(\zeta_i z)&=\Psi_F(\zeta_i\Nm_{L_i/F}(z))
                      &&(z\in\pp_{L_i}^{\ceil{T_i/2}}).
 \end{aligned}
\end{equation}
The first identity is \SMcite{D:NM:lower}. The second follows from
\eqref{eq:dur-norm-trace} with additive scale $\alpha\zeta_i$.
Both identities are formalized by
\lean{Dyadic/Nonmaximal/LowerCharacters.lean}{214}{lowerCharacters}.

Put $m_1=2T_2-1$, $m_2=m_3=T_1+T_2-1$, and
$s_i=\ceil{m_i/2}$. Define the prescribed unit characters on
$U_{L_i}^{s_i}$ by
\begin{equation}\label{eq:dnm-model-charts}
 R_i(1-z)=\Psi_{L_i}(a_i z)\qquad(z\in\pp_{L_i}^{s_i}).
\end{equation}
Their conductors are $m_i$.
For the nontrivial $\sigma\in\Gal(L_1/F)$, put
$\nu_1=\omega_2\circ\Nm_{L_1/F}$.
For $v\in L_1^\times$ with $\sigma(v)/v\in U_{L_1}^{s_1}$, one has
$R_1(\sigma(v)/v)=\nu_1(v)$.
This is \SMcite{D:NM:commmodel}, formalized by
\lean{Dyadic/Nonmaximal/Commutator.lean}{584}{commutator}.

The symmetric identity \eqref{eq:biquadratic-e2} gives,
for $z\in\pp_K^{T_2}$, the common norm pullback
\[
 R_i(\Nm_{K/L_i}(1-z))
   =\Psi_F(-e_2(Y(1-z))+e_2(Y)).
\]
This is \SMcite{D:NM:compat} and
\lean{Dyadic/Nonmaximal/Compatibility.lean}{203}{compatibility}.
As in Section~\ref{subsec:deq-stationary}, the characters $R_i$
extend to a primitive compatible family $(\chi_i^0)$ of conductors
$m_i$. This is \SMcite{D:NM:core-exists}, formalized by
\lean{Dyadic/Nonmaximal/Models.lean}{349}{models}.

The analogue of \eqref{eq:deq-twist-conductors} is
\SMcite{D:NM:realize}: there are a quasi-character $\lambda$ of
$F^\times$ and primitive compatible characters $\chi_i$ of conductors
$m_i$ such that, for $i=1,2,3$,
\begin{equation*}
 \begin{gathered}
  \theta_i=\chi_i(\lambda\circ\Nm_{L_i/F}),\qquad
  \chi_i|_{U_{L_i}^{s_i}}=R_i,\\
  m_{L_i}(\theta_i)=
  \begin{cases}
   m_i,&n\le T_2+r_1-1,\\
   2n-T_i,&n\ge T_2+r_1,
  \end{cases}
  \qquad n=m_F(\lambda).
 \end{gathered}
\end{equation*}
This statement is formalized by
\lean{Dyadic/Nonmaximal/Twist.lean}{206}{twist}.

\subsection{Exact norm normalization and stationary coefficients}
\label{subsec:dnm-normalization}

By \SMcite{D:NM:freedom}, the evaluations of $\omega_i$ and $R_i$
in \eqref{eq:dnm-lower-charts} and \eqref{eq:dnm-model-charts},
respectively, remain valid after replacing $\alpha$ by $\alpha u$
for any $u\in U_F^{T_2-1}$.
Since $\omega_3$ is nontrivial on $U_F^{T_2-1}$, for every
$\gamma\in F^\times$ one may choose $u$ so that
\begin{equation}\label{eq:exact-norm-normalization}
 \frac{\gamma}{\alpha u}=\Nm_{L_3/F}(X)
                        \qquad\text{for some }X\in L_3^\times.
\end{equation}
We call this \emph{exact norm normalization}. It is formalized by
\lean{Dyadic/Nonmaximal/Normalization.lean}{176}{normalization} and
\lean{Stationary/NormalizationFreedom.lean}{99}{normalizationFreedom}.

For $n\le T_2+r_1-1$, the proof of \SMcite{D:NM:minimal-origin}
chooses $u\in U_F^{T_2-1}$ and $X\in\pp_{L_3}^{1-r_1}$.
Replace $\alpha$ by $\alpha u$, using the new scale in every $\Psi_L$.
With $A=\Nm_{L_3/F}(X)$, these choices satisfy
\[
 \lambda(1-z)=\Psi_F(Az)\qquad(z\in\pp_F^{r_1+r_2}).
\]
Take $u=1$ and $X=0$ when $n\le r_1+r_2$.
Put $C=Y-X$. Then $v_K(C)=-t_1$, and its traces to $L_1,L_2$
are nonzero.

For $n\ge T_2+r_1$, choose $\gamma$ for $\lambda$ as in
\eqref{eq:stable-base-coefficient}. Choose $u\in U_F^{T_2-1}$ and
$X\in L_3^\times$ by \eqref{eq:exact-norm-normalization}, and
replace $\alpha$ by $\alpha u$ in every $\Psi_L$.
Then $v_{L_3}(X)=T_2-n$. Put $C=Y-X$; now
$v_K(C)=2(T_2-n)$.

For $i=1,2$, the stationary coefficient of $\theta_i$ is
$\Nm_{K/L_i}(C)$ in both ranges. The coefficient used for $\omega_i$
is the norm of the trace of $C$ in the minimal range, but that of $Y$
in the higher range:
\begin{equation}\label{eq:dnm-stationary}
 \begin{gathered}
  \theta_i(1-z)=\Psi_{L_i}(\Nm_{K/L_i}(C)z),
  \qquad z\in\pp_{L_i}^{\ceil{m_{L_i}(\theta_i)/2}},\\[1mm]
  \omega_i(1-z)=
  \begin{cases}
   \Psi_F(\Nm_{L_i/F}(\Tr_{K/L_i}(C))z),&n\le T_2+r_1-1,\\
   \Psi_F(\zeta_i z),&n\ge T_2+r_1,
  \end{cases}
  \qquad z\in\pp_F^{\ceil{T_i/2}}.
 \end{gathered}
\end{equation}
The depth in the $\theta_i$-evaluation is $s_i$ in the minimal range
and $n-r_i$ in the higher range. The minimal-range assertions are
\SMcite{D:NM:minimal-origin}, formalized by
\lean{Dyadic/Nonmaximal/MinimalOrigin.lean}{592}{minimalOrigin}.
When $t_1=t_2$, both minimal-range evaluations also hold for $i=3$.

For the higher range, the bounds of \SMcite{D:NM:bounds} permit
applying \cite[(\SMLnum{D:NM:Qphase})]{SML} and the subsequent
norm-expansion step on $\pp_{L_i}^{n-r_i}$. Multiplying the resulting
evaluation of $\lambda\circ\Nm_{L_i/F}$ by
\eqref{eq:dnm-model-charts} gives the $\theta_i$-evaluation in
\eqref{eq:dnm-stationary}. This construction is formalized by
\lean{Dyadic/Nonmaximal/Higher.lean}{378}{higher_stationary}.

Put $p_X=\Tr_{L_3/F}(X)$. The difference between the two
$\omega_i$-coefficients is given by \SMcite{D:NM:bounds}:
\begin{equation}\label{eq:dnm-trace-norm-correction}
 \Nm_{L_i/F}(\Tr_{K/L_i}(C))
   =\zeta_i+p_X^2+2(1+d)p_X,\qquad i=1,2.
\end{equation}
This identity and the valuation bounds are formalized by
\lean{Dyadic/Nonmaximal/CommonError.lean}{516}{commonError}.

\subsection{Comparison at minimal and higher conductor}

Use $C$ and the stationary coefficients in \eqref{eq:dnm-stationary},
and let $g_{\theta_i},g_{\omega_i}$ be their normalized critical sums
as in \eqref{eq:critical-function}. Since $T_i$ is even,
$g_{\omega_i}=1$ in both ranges. Multiplying \eqref{eq:normalized-factor}
for $\theta_i$ and $\omega_i$, as in the derivation of
\eqref{eq:deq-induction-cases}, now using
\eqref{eq:dnm-trace-norm-correction}, gives, for $i=1,2$,
\begin{equation}\label{eq:dnm-induction-cases}
 \begin{aligned}
  \mathcal A_{L_i}(\theta_i)
  &=D_\theta(-\alpha^{-1})\Theta(C)^{-1}\\
  &\quad\times
  \begin{cases}
   \Psi_F(-e_2(C))g_{\theta_i},&n\le T_2+r_1-1,\\
   \Psi_F\bigl(-e_2(C)+p_X^2+2(1+d)p_X\bigr),&n\ge T_2+r_1.
  \end{cases}
 \end{aligned}
\end{equation}
In the second range the conductors $2n-T_i$ are even, so
$g_{\theta_i}=1$ as well.

For $n\le T_2+r_1-1$, the critical functions have a common pullback
along the norm maps on unit quotients, with the same formulas for
these maps as in Section~\ref{subsec:deq-comparison}.
The bijective-map argument for $t_1<t_2$, and
\eqref{eq:missing-coset} for $t_1=t_2$, therefore give
\[
 \begin{aligned}
  g_{\theta_1}&=g_{\theta_2}&&(t_1<t_2),\\
  g_{\theta_1}&=g_{\theta_2}=g_{\theta_3}&&(t_1=t_2).
 \end{aligned}
\]
These two cases are formalized by
\lean{Dyadic/Nonmaximal/UnequalBreaks.lean}{419}{unequalBreaks}
and \lean{Dyadic/Nonmaximal/EqualBreaks.lean}{1421}{equalBreaks},
respectively. The first case of \eqref{eq:dnm-induction-cases}
now gives the minimal comparison
\SMcite{D:NM:minimal-two}.

For $n\ge T_2+r_1$, the second case of \eqref{eq:dnm-induction-cases}
is independent of $i$,
including the factor $\Psi_F(p_X^2+2(1+d)p_X)$. Hence
$\mathcal A_{L_1}(\theta_1)=\mathcal A_{L_2}(\theta_2)$.
This is \SMcite{D:NM:higher}, formalized by
\lean{Dyadic/Nonmaximal/Higher.lean}{608}{higher}.

The two ranges are consecutive, and relabeling gives the remaining
comparison. Their assembly is
\lean{Dyadic/Nonmaximal/Comparison.lean}{473}{comparison}.

\section{Totally ramified mixed characteristic with maximal break}\label{sec:dyadic-maximal}
Let $F/\Q_2$ be finite with residue field $k$, and put $e=v_F(2)$.
Let $K/F$ be totally ramified biquadratic, and label its quadratic
intermediate fields $L_i$ with breaks
\[
 t_1=2r_1-1,\qquad t_2=t_3=2e,\qquad 1\le r_1\le e.
\]
Choose a uniformizer $\pi$ of $K$, and use the compatible
uniformizers
\begin{equation*}
 \Pi_i=\Nm_{K/L_i}(\pi),\qquad
 \varpi=\Nm_{K/F}(\pi)=\Nm_{L_i/F}(\Pi_i).
\end{equation*}
Let $\theta_i$ be a primitive compatible family with common pullback
$\Theta$, and let $\omega_i$ be the nontrivial norm character of $L_i/F$.
Then $T_1=2r_1$ and $T_2=T_3=2e+1$ are the norm-character
conductors. The proof uses the same norm normalization as
Section~\ref{sec:dyadic-nonmaximal}, but the odd conductor
$T_2$ changes which residue-field factors occur. At minimal conductor,
we compare the normalized sums for $\theta_1$ and $\omega_2$.
At higher conductor, we prove that the normalized sums for
$\theta_2$ and $\omega_2$ are inverse.  

\subsection{Kummer generators, model characters, and conductor parity}

Choose Kummer generators $S\in L_1$ and $R\in L_2$ with
$S^2,R^2\in F$.  The Kummer alignment of \SMcite{D:MX:align}
satisfies
\[
 \epsilon:=S^2-1,\qquad
 v_F(\epsilon)=v_F(R^2)=2e-2r_1+1,\qquad
 \epsilon/R^2\in U_F^{r_1}.
\]
Put $L_3=F(SR)$ and $Y=(S+R-1)/2$, and write
$a_i=\Nm_{K/L_i}(Y)$, $h_i=\Tr_{K/L_i}(Y)$ and
$\zeta_i=\Nm_{L_i/F}(h_i)$.  Then
\[
 \begin{aligned}
  h_1&=S-1,&\zeta_1&=-\epsilon,\\
  h_2&=R-1,&\zeta_2&=1-R^2,\\
  h_3&=-1,&\zeta_3&=1.
 \end{aligned}
\]
with $v_K(Y)=v_{L_i}(a_i)=-t_1$ and
$\Tr_{K/F}(Y)=-2$. These data are
\SMcite{D:MX:origin-identities}, formalized by
\lean{Dyadic/Maximal/Origin.lean}{902}{dyadicMaximal_origin_data}.

Choose $\alpha\in F^\times$ so that
$Z=1$ is a stationary coefficient for $\omega_3$ in
\eqref{eq:stationary-coefficient}, and put $\Psi_L(x)=\psi_L(\alpha x)$. The largest trivial ideals of
$\Psi_F$ and $\Psi_{L_i}$ are $\pp_F^{T_2}$ and
$\pp_{L_i}^{J_i}$, where
\[
 J_1=2T_2-T_1,\qquad J_2=J_3=T_2.
\]
For $i=1,2,3$,
\begin{equation}\label{eq:dmx-lower-charts}
 \begin{aligned}
 \omega_i(1-z)&=\Psi_F(\zeta_i z)
                      &&(z\in\pp_F^{\ceil{T_i/2}}),\\
 \Psi_{L_i}(\zeta_i z)&=\Psi_F(\zeta_i\Nm_{L_i/F}(z))
                      &&(z\in\pp_{L_i}^{\ceil{T_i/2}}).
 \end{aligned}
\end{equation}
The first identity is \SMcite{D:MX:lower}. The second follows from
\eqref{eq:dur-norm-trace} with additive scale $\alpha\zeta_i$.
Put $m_1=2T_2-1$, $m_2=m_3=T_1+T_2-1$, and
$s_i=\ceil{m_i/2}$. Prescribe model unit characters on
$U_{L_i}^{s_i}$ by
\begin{equation}\label{eq:dmx-model-charts}
 R_i(1-z)=\Psi_{L_i}(a_i z)\qquad(z\in\pp_{L_i}^{s_i}).
\end{equation}
By \SMcite{D:MX:Rcompat}, the three norm pullbacks
$R_i\circ\Nm_{K/L_i}$ are defined and agree on $U_K^{T_2}$.

Starting with an extension of $R_2$ to $L_2^\times$, the proof of
\SMcite{D:MX:cores} constructs a primitive compatible family
$(\chi_i^0)$ extending the $R_i$, with conductors $m_i$.
This construction is formalized by
\lean{Dyadic/Maximal/Models.lean}{1531}{models}.

The analogue of \eqref{eq:deq-twist-conductors} is
\SMcite{D:MX:realize}: after conjugating $\theta_1$ if necessary,
there are a quasi-character $\lambda$ of $F^\times$ and primitive
compatible characters $\chi_i$ of conductors $m_i$ such that,
for $i=1,2$,
\[
 \begin{gathered}
  \theta_i=\chi_i(\lambda\circ\Nm_{L_i/F}),\qquad
  \chi_i|_{U_{L_i}^{s_i}}=R_i,\\
  m_{L_i}(\theta_i)=
  \begin{cases}
   m_i,&n\le T_2+r_1-1,\\
   2n-T_i,&n\ge T_2+r_1,
  \end{cases}
  \qquad n=m_F(\lambda).
 \end{gathered}
\]
The conjugation preserves the local constant and common norm pullback.
This statement is formalized by
\lean{Dyadic/Maximal/Twist.lean}{466}{twist}.
The parity separation is important: the model conductor over $L_1$ is
odd and its larger twist conductor is even, whereas over $L_2$ the
model conductor is even and its larger twist conductor is odd.

\subsection{Exact norm normalization and stationary coefficients}
\label{subsec:dmx-normalization}

By \SMcite{D:MX:freedom}, the evaluations of $\omega_i$ and $R_i$
in \eqref{eq:dmx-lower-charts} and \eqref{eq:dmx-model-charts}
remain valid after replacing $\alpha$ by $\alpha u$ for any
$u\in U_F^{T_2-1}$. The exact-norm argument in
Section~\ref{subsec:dnm-normalization} therefore gives
\eqref{eq:exact-norm-normalization} in this case as well.

For $n\le T_2+r_1-1$, the proof of \SMcite{D:MX:minimal-origin}
chooses $u\in U_F^{T_2-1}$ and $X\in\pp_{L_3}^{1-r_1}$.
Replace $\alpha$ by $\alpha u$, using the new scale in every $\Psi_L$.
With $A=\Nm_{L_3/F}(X)$, these choices satisfy
\[
 \lambda(1-z)=\Psi_F(Az)\qquad(z\in\pp_F^{e+r_1}).
\]
Take $u=1$ and $X=0$ when $n\le e+r_1$.
Put $C=Y-X$. Then $v_K(C)=-t_1$, and its traces to $L_1,L_2$
are nonzero.

For $n\ge T_2+r_1$, choose $\gamma$ for $\lambda$ as in
\eqref{eq:stable-base-coefficient}. Choose $u\in U_F^{T_2-1}$ and
$X\in L_3^\times$ by \eqref{eq:exact-norm-normalization}, and
replace $\alpha$ by $\alpha u$ in every $\Psi_L$.
Then $v_{L_3}(X)=T_2-n$. Put $C=Y-X$; now
$v_K(C)=2(T_2-n)$.

As in \eqref{eq:dnm-stationary}, the coefficient of $\theta_i$ is
$\Nm_{K/L_i}(C)$ in both ranges, whereas the coefficient for
$\omega_i$ uses the trace of $C$ in the minimal range and that of $Y$
in the higher range. For $i=1,2$,
\begin{equation}\label{eq:dmx-stationary}
 \begin{gathered}
  \theta_i(1-z)=\Psi_{L_i}(\Nm_{K/L_i}(C)z),
  \qquad z\in\pp_{L_i}^{\ceil{m_{L_i}(\theta_i)/2}},\\[1mm]
  \omega_i(1-z)=
  \begin{cases}
   \Psi_F(\Nm_{L_i/F}(\Tr_{K/L_i}(C))z),&n\le T_2+r_1-1,\\
   \Psi_F(\zeta_i z),&n\ge T_2+r_1,
  \end{cases}
  \qquad z\in\pp_F^{\ceil{T_i/2}}.
 \end{gathered}
\end{equation}
The depth in the $\theta_i$-evaluation is $s_i$ in the minimal range;
in the higher range it is $n-r_1$ for $L_1$ and $n-e$ for $L_2$.
The minimal-range assertions are \SMcite{D:MX:minimal-origin},
formalized by
\lean{Dyadic/Maximal/MinimalOrigin.lean}{491}{minimalOrigin}.
The higher-range construction is \SMcite{D:MX:higher-coeff},
formalized by
\lean{Dyadic/Maximal/HigherCoefficients.lean}{437}{higherCoefficients}.

Put $p_X=\Tr_{L_3/F}(X)$. The difference between the two
$\omega_i$-coefficients is given by \SMcite{D:MX:error-estimates}:
\begin{equation}\label{eq:dmx-trace-norm-correction}
 \Nm_{L_i/F}(\Tr_{K/L_i}(C))=\zeta_i+\Delta,\qquad
 \Delta=p_X^2+2p_X,\qquad i=1,2.
\end{equation}
This identity and the valuation bounds are formalized by
\lean{Dyadic/Maximal/CommonError.lean}{960}{commonError}.

\subsection{Comparison at minimal and higher conductor}

Use $C$ and the stationary coefficients in \eqref{eq:dmx-stationary},
and let $g_{\theta_i},g_{\omega_i}$ be their normalized critical sums
as in \eqref{eq:critical-function}. In the minimal range,
\eqref{eq:common-origin} applies. In the higher range, use
\eqref{eq:normalized-factor}, \eqref{eq:biquadratic-e2}, and
\eqref{eq:dmx-trace-norm-correction}. Thus, for $i=1,2$,
\begin{equation}\label{eq:dmx-induction-cases}
 \begin{aligned}
  \mathcal A_{L_i}(\theta_i)
  &=D_\theta(-\alpha^{-1})\Theta(C)^{-1}\\
  &\quad\times
  \begin{cases}
   \Psi_F(-e_2(C))g_{\theta_i}g_{\omega_i},&n\le T_2+r_1-1,\\
   \Psi_F(-e_2(C)+\Delta)g_{\theta_i}g_{\omega_i},&n\ge T_2+r_1.
  \end{cases}
 \end{aligned}
\end{equation}

For $n\le T_2+r_1-1$, only $\theta_1$ and $\omega_2$ have odd
conductor, so the first case of \eqref{eq:dmx-induction-cases}
reduces the comparison to $g_{\theta_1}=g_{\omega_2}$.
Let $H_{\theta_1},H_{\omega_2}$ be the critical functions for these
coefficients with critical elements $\Pi_1^{2e}$ and $\varpi^e$.
Put $C_z=C(1+\pi^{2e}z)$ for $z\in\OO_K$.
Let $\rho_{\theta_1},\rho_{\omega_2}:k\to k$ be the maps induced
on the critical unit quotients by
$\Nm_{K/L_1}(C_z)/\Nm_{K/L_1}(C)$ and
$\Nm_{L_2/F}(\Tr_{K/L_2}(C_z))/\Nm_{L_2/F}(\Tr_{K/L_2}(C))$,
respectively. Both maps are bijective. Applying the character-product
identity of Section~\ref{subsec:common-coefficients} to $C,C_z$ gives
\[
 H_{\theta_1}(\rho_{\theta_1}(\bar z))
 =H_{\omega_2}(\rho_{\omega_2}(\bar z))\qquad(z\in\OO_K).
\]
Reindexing gives $g_{\theta_1}=g_{\omega_2}$ and proves
\SMcite{D:MX:minimal}, formalized by
\lean{Dyadic/Maximal/MinimalComparison.lean}{753}{minimalComparison}.

For $n\ge T_2+r_1$, the conductors of $\theta_1,\omega_1$ are even,
whereas those of $\theta_2,\omega_2$ are odd. The second case of
\eqref{eq:dmx-induction-cases} therefore reduces the comparison to
$g_{\theta_2}g_{\omega_2}=1$.
Let $H_{\theta_2},H_{\omega_2}$ be the critical functions for the
higher-range coefficients in \eqref{eq:dmx-stationary}, with critical
elements $\Pi_2^{n-e-1}$ and $\varpi^e$, respectively.
For $z\in\OO_K$, put
\[
 V=Y\pi^{2e}z,\qquad C_z=C+V,\qquad Y_z=Y+V.
\]
The induced maps $\rho_{\theta_2},\rho_{\omega_2}:k\to k$ are
\[
 \begin{aligned}
 \rho_{\theta_2}(\bar z)
   &=\overline{\frac{\Nm_{K/L_2}(1+V/C)-1}{\Pi_2^{n-e-1}}},\\
 \rho_{\omega_2}(\bar z)
   &=\overline{\frac{\Nm_{L_2/F}(1+\Tr_{K/L_2}(V)/h_2)-1}{\varpi^e}}.
 \end{aligned}
\]
The valuation and ramification checks making these maps well defined
and bijective are formalized by
\lean{Dyadic/Maximal/HigherComparison.lean}{407}{higherResidueBijections}.
For $C_z,Y_z$ defined above, the symmetric identity
\eqref{eq:biquadratic-e2} and compatibility give
\[
 H_{\theta_2}(\rho_{\theta_2}(\bar z))
 H_{\omega_2}(\rho_{\omega_2}(\bar z))=1\qquad(z\in\OO_K).
\]
Reindexing by the two bijections yields
$g_{\theta_2}=g_{\omega_2}^{-1}$.
This proves \SMcite{D:MX:higher}, formalized by
\lean{Dyadic/Maximal/HigherComparison.lean}{895}{higherComparison}.
Interchanging $L_2,L_3$ and using inducing-character conjugacy gives
the original family; the final declaration is
\lean{Dyadic/Maximal/Comparison.lean}{924}{comparison}.

\section{Assembly of the formal proof}\label{sec:completion}
\subsection{Exhaustion of cases}

\SMcite{U:dispatch} states that every Galois extension $K/F$ with
group $C_\ell^2$ satisfies the hypotheses of one of the comparison
theorems treated in Sections~\ref{sec:tame}--\ref{sec:dyadic-maximal}.
This classification is formalized by
\lean{Classification/Exhaustion.lean}{646}{exhaustion}.

\subsection{Auxiliary results}
\label{subsec:auxiliary}

Classical auxiliary results proved in
\cite[Appendices~A--D]{SML} are also formalized in Lean.

Appendix A gives the leading Gauss-sum congruence and finite
interpolation identities, formalized by
\lean{Gauss/LeadingTerm.lean}{288}{Gauss.leadingTerm}
and \lean{Finite/Interpolation.lean}{367}{interpolation}.

Appendix B proves the local trace--residue identity and the
logarithmic derivative formula for norms, formalized by
\lean{Residues/TraceCompatibility.lean}{704}{traceCompatibility}
and \lean{Residues/LogNorm.lean}{393}{logNorm}.

Appendix C contains Hilbert~90, continuous character extension,
and the first wild ramification term, formalized by
\lean{Local/Hilbert90.lean}{73}{hilbert90},
\lean{Local/CharacterExtension.lean}{193}{characterExtension},
and \lean{Ramification/LeadingTerm.lean}{358}{Ramification.leadingTerm}.

Appendix D contains Newton approximation and the distinction
of norm subgroups of degree-$\ell$ intermediate fields, formalized by
\lean{Local/Newton.lean}{402}{newton}
and \lean{Ramification/NormSeparation.lean}{1557}{normSeparation}.

\section{Source and reproducibility}\label{sec:artifact}
The declaration links refer to the formalization of the Second Main
Lemma~\cite{SMLCode} at
revision \code{5b55d3f6e7c053f733afb1be523e532d212d94d0} of
\url{https://github.com/fukubillueda-web/epsilon_SML}.
Its pinned configuration uses Lean \code{4.32.2}, Mathlib revision
\code{905b95818eb32af7874a58b427f50c1711a5e96c}, and the formalization of the
First Main Lemma~\cite{FMLCode} at revision
\code{b91ad93900a5e6d9207a4e6166ac94bb3bbcf3ed}.
The repositories are distributed under the Apache License~2.0.

With Git and the specified Lean toolchain installed, the following commands
build the revision and invoke the independent checker on the main module:
\begin{lstlisting}[language={}]
git clone https://github.com/fukubillueda-web/epsilon_SML.git
cd epsilon_SML
git checkout v1.0.0
lake exe cache get
lake build
lake env leanchecker --fresh LanglandsSecondMainLemma.Main
\end{lstlisting}
The transitive axiom dependencies of the exported theorem can be inspected
with:
\begin{lstlisting}[language={}]
cat > VerifySML.lean <<'EOF'
import LanglandsSecondMainLemma.Main
#print axioms LanglandsSecondMainLemma.secondMainLemma
EOF
lake env lean VerifySML.lean
\end{lstlisting}
The file \code{SOURCE_MAP.md} at the pinned revision uses the numbering
of the reference manuscript bundled with that revision.
The declaration links in this paper identify the Lean counterparts of
the statements presented here.

\section*{Acknowledgments}
ChatGPT assisted with transferring formulas and citations from the companion
paper and with adding references to the corresponding Lean source, and
suggested improvements to the wording and organization of the paper.

\end{document}